\documentclass[11pt]{article}
\usepackage{graphicx}
\usepackage{amsmath}
\usepackage{amssymb}
\usepackage{amsthm}
\usepackage[titletoc]{appendix}
\usepackage{url}
\newcommand\al\alpha
\newcommand\be\beta
\newcommand\de\delta
\newcommand\ep\varepsilon
\newcommand\tha\theta
\newcommand\ka\kappa
\newcommand\la\lambda
\newcommand\om\omega

\newcommand\iy\infty
\newcommand\pa\partial

\newcommand{\hyp}[5]{\,\mbox{}_{#1}F_{#2}\!\left(\genfrac{}{}{0pt}{}{#3}{#4};#5\right)}

\numberwithin{equation}{section}

\newtheorem{theorem}{Theorem}

\newtheorem{lemma}[theorem]{Lemma}

\newtheorem{Remark}[theorem]{Remark}

\begin{document}

\title{Exact formulas for the fourth and fifth cumulant of the Rosenblatt distribution.}
\author{Enno Diekema \footnote{email address: e.diekema@gmail.com}}
\date{}
\maketitle

\begin{abstract}
\noindent
In an earlier version of this paper \cite{6} a formula for the 4th cumulant of the Rosenblatt distribution was derived. In this paper formulas for the 3rd, 4th and 5th cumulant are derived using two methods. In the first method the necessary integrals are determined directly. The second method is the one used by Veillette and Taqqu. This is a recurring method. Finally both methods are discussed.
\end{abstract}

\section{Introduction}
\setlength{\parindent}{0cm}
\noindent

For an introduction of the Rosenblatt distribution we follow Tudor \cite{3} and Veilette and Taqqu \cite{4}. We have to concider  a stationary Gaussian sequence $X_i, i=1,2,...$ which has a covariance structure of the form $E[X_0 X_k]=L(k)\, k^{-d}$ as $k \rightarrow \infty$ with $0 \leq d \leq 0.5$ and $L(k)$ is a slowly varying function at infinity. Using the transformation $Y_i=(X_i)^2-1$ one can define a sequence of normalized sums
\[
Z_d^n=\dfrac{\sigma(d)}{n^{1-d}}\sum_{i=1}^n Y_i
\]
Here $\sigma(d)$ is a normalizing constant. This is given by
\[
\sigma(d)=\left(\dfrac{1}{2}(1-2d)(1-d)\right)^{1/2}
\]
The sequence $Z_d^n$ tends to a non-Gaussian limit $Z_d$ as $n \rightarrow \infty$ with variance $1$. The limit has a distribution named the Rosenblatt distribution. The characteristic function of $Z_d$ is given by a power series which is convergent only near the origin:
\begin{equation}
\varphi(\theta)=\exp\left( \dfrac{1}{2}\sum_{k=2}^\infty\big(2\, i\, \theta\, \sigma(d)\big)^k\, \dfrac{c_k}{k}\right)
\label{1.1a}
\end{equation}
where
\begin{equation}
c_k=\int_0^1\int_0^1...\int_0^1|x_1-x_2|^{-d}|x_2-x_3|^{-d}...|x_{k-1}-x_k|^{-d}|x_k-x_1|^{-d}dx_1\, dx_2...dx_k
\label{1.2}
\end{equation}

When $d=0$ the Rosenblatt distribution function is a standard chi-squared distribution with mean $0$ and variance $1$. When $d=0.5$ the distribution is a $N(0,1)$ Gauss distribution.

\

From \eqref{1.1a} the cumulants $\kappa_k(d)$ of the Rosenblatt distribution are $\kappa_1(0)=0$ and for $k \geq 2$
\begin{equation}
\kappa_k(d)=2^{k-1}(k-1)!\big(\sigma(d)\big)^k c_k=
2^{k-1}(k-1)!\left(\dfrac{1}{2}(1-2d)(1-d)\right)^{k/2} c_k
\label{1.3}
\end{equation}
where $c_k$ is given by \eqref{1.2}. The difficulties of the integral \eqref{1.2} lies in the absolute value of the different factors. In \cite{4} the authors claims that the cumulants for $k > 3$ could not be found. They describe a method with recurrence equations so there are less integrals. Then they approximate the integrals numerically and so they come to a table for the cumulants \cite[Table 4]{5}. In this paper we derive formulas for $\kappa_4(d)$ and $\kappa_5(d)$ by computing the integrals with two methods. The first method computes all the integrals directly. The second method of Veillette and Taqqu computes the integral recursive.

\

The second cumulant of the Rosenblatt distribution is given by $\kappa_2(d)=1$. In Section 3 we start with the direct computation of $\kappa_3(d)$. This third cumulant is known, but here we use  a method which is also used for the derivation of the higher cumulants. Section 4 treats the derivation of a formula for the fourth cumulant. In Section 5 we derive a formula for the fifth cumulant. The results are checked with the results from  \cite[Table 4]{5}. In Section 6 we describe the method of Veillette and Taqqu and apply this method to compute the integrals required for the fourth cumulant exactly. In Section 7 we used this method for the derivation of the fifth cumulant. In Section 8 we discuss both methods. The Appendix contains a number of Thomae transformations for the $_3F_2$ hypergeometric function \cite{1}.

\

\textbf{Remark 1:}\  In the whole paper integrals and summations are interchanged when needed. In all cases it can be shown that this is allowed by the dominant convergence theorem, but the proofs here are omitted while they are irrelevant for the course of the content of the paper.
\vspace{0.3 cm}

\textbf{Remark 2:}\  The simple integrals can usually be found in \cite{2}. Of course the integrals can be computed by hand with standard methods, but why make it difficult when it can be done easily as they can also be determined with Mathematica (a program of the Wolfram company). 
\vspace{0.3 cm}

\textbf{Remark 3:}\ The properties of the used Pochhammer symbols and the Gamma functions are listed in the Appendix of \cite{1}. We use very often the property $\Gamma(a+k)=\Gamma(a)\, (a)_k$ where $(a)_k$ is the Pochhammer symbol.
\vspace{0.3 cm}

\textbf{Remark 4:}\  For the derivation of the integrals we use frequently (when it is allowed) the transformation
\begin{equation}
\int_0^a f(y)\left(\int_0^y g(x)dx\right)dy=\int_0^a g(x)\left(\int_x^a f(y)dy\right)dx
\qquad a>y>x>0
\label{1.1}
\end{equation}
\vspace{0.3 cm}

\textbf{Remark 5:}\ Another transformation we use frequently is the substitution: $y_5=y_3\, z_3$, $y_3=y_2\, z_3$ and $y_2=y_1\, z_2$.

\section{Main result}
The main results of this paper are the formulas of the fourth and fifth cumulant of the Rosenblatt distribution as function of the parameter $d$ with $0 \leq d \leq 0.5$.

\

\begin{align*}
\kappa_4(d)&=\dfrac{12(1-2d)^2}{(1-d)(3-4d)}
\left(2\dfrac{\Gamma(1-d)\Gamma(2-d)^2}{\Gamma(3-3d)}+
\dfrac{\Gamma(1-d)^2\Gamma(2-d)^2}{\Gamma(2-2d)^2}+\hyp32{1,d,2-2d}{2-d,3-2d}{1}
\right) \\
\kappa_5(d)&=3840\left(\dfrac{1}{2}(1-2d)(1-d)\right)^{5/2} \dfrac{\Gamma(4-5d)}{\Gamma(6-5d)} S
\end{align*}
with
\begin{align*}
S&=\dfrac{\Gamma(1-d)^4}{\Gamma(4-4d)}+ 
\dfrac{\Gamma(1-d)^3}{(1-d)\Gamma(2-2d)}\dfrac{\Gamma(3-3d)}{\Gamma(4-4d)}
\hyp32{d,1-d,3-3d}{2-d,4-4d}{1}+ \\
&+\dfrac{\Gamma(1-d)^2\Gamma(2-2d)}{(1-d)\Gamma(4-4d)}\hyp32{d,1-d,2-2d}{2-d,4-4d}{1}
+\dfrac{\Gamma(1-d)^2}{3(1-d)\Gamma(3-2d)}\hyp32{1,d,3-3d}{3-2d,4-3d}{1}+ \\
&+\dfrac{2\Gamma(1-d)^2}{3(1-d)^2\Gamma(2-2d)}\hyp32{1,d,3-3d}{2-d,4-3d}{1}
+\dfrac{\Gamma(1-d)^2}{(1-d)\Gamma(3-2d)}\hyp32{1,d,2-2d}{2-d,3-2d}{1}
\end{align*}
\vspace{0.1 cm}

\section{Direct derivation of the third cumulant}
Although the third cumulant is known \cite{4}, it is derived here in a direct way. This method is also used for the computation of the fourth and fifth cumulants in the next sections.

For the third cumulant \eqref{1.3} gives
\begin{equation}
\kappa_3(d)=8\left(\dfrac{1}{2}(1-2d)(1-d)\right)^{3/2}c_3 \qquad 0 \leq d \leq 0.5
\label{3.1}
\end{equation}
For $c_3$ we get from \eqref{1.2}
\begin{equation}
c_3=\int_0^1\int_0^1\int_0^1 |x_1-x_2|^{-d}|x_2-x_3|^{-d}|x_3-x_1|^{-d}dx_1\, dx_2\, dx_3
\label{3.2}
\end{equation}
Following the definition of the absolute function
\begin{align*}
&|x|=x \qquad\ \ x>0 \\
&|x|=-x \qquad x<0
\end{align*}
and application to \eqref{3.2} gives
\begin{align*}
&|x_1-x_2|=x_1-x_2 \qquad x_1>x_2 \\
&|x_1-x_2|=x_2-x_1 \qquad x_2>x_1 \\
&|x_2-x_3|=x_2-x_3 \qquad x_2>x_3 \\
&|x_2-x_3|=x_3-x_2 \qquad x_3>x_2 \\
&|x_3-x_1|=x_3-x_1 \qquad x_3>x_1 \\
&|x_3-x_1|=x_1-x_3 \qquad x_1>x_3
\end{align*}
Each factor in \eqref{3.2} gives two possibilities and there are three factors so application to the integration domain gives eight possible conditions:
\begin{align*}
&x_1>x_2 \qquad x_2>x_3 \qquad x_3>x_1 \\
&x_1>x_2 \qquad x_2>x_3 \qquad x_1>x_3 \qquad\qquad 0<x_3<x_2<x_1<1 \\
&x_1>x_2 \qquad x_3>x_2 \qquad x_3>x_1 \qquad\qquad 0<x_2<x_1<x_3<1 \\
&x_1>x_2 \qquad x_3>x_2 \qquad x_1>x_3 \qquad\qquad 0<x_2<x_3<x_1<1 \\
&x_2>x_1 \qquad x_2>x_3 \qquad x_3>x_1 \qquad\qquad 0<x_1<x_3<x_2<1 \\
&x_2>x_1 \qquad x_2>x_3 \qquad x_1>x_3 \qquad\qquad 0<x_3<x_1<x_2<1 \\
&x_2>x_1 \qquad x_3>x_2 \qquad x_3>x_1 \qquad\qquad 0<x_1<x_2<x_3<1 \\
&x_2>x_1 \qquad x_3>x_2 \qquad x_1>x_3 
\end{align*}
The first and the last conditions give contradictions. So there remains six regions which are shown in Figure 1. Together these give a cubic region.
\begin{figure}[ht]
	\centering
	\includegraphics[width=340pt]{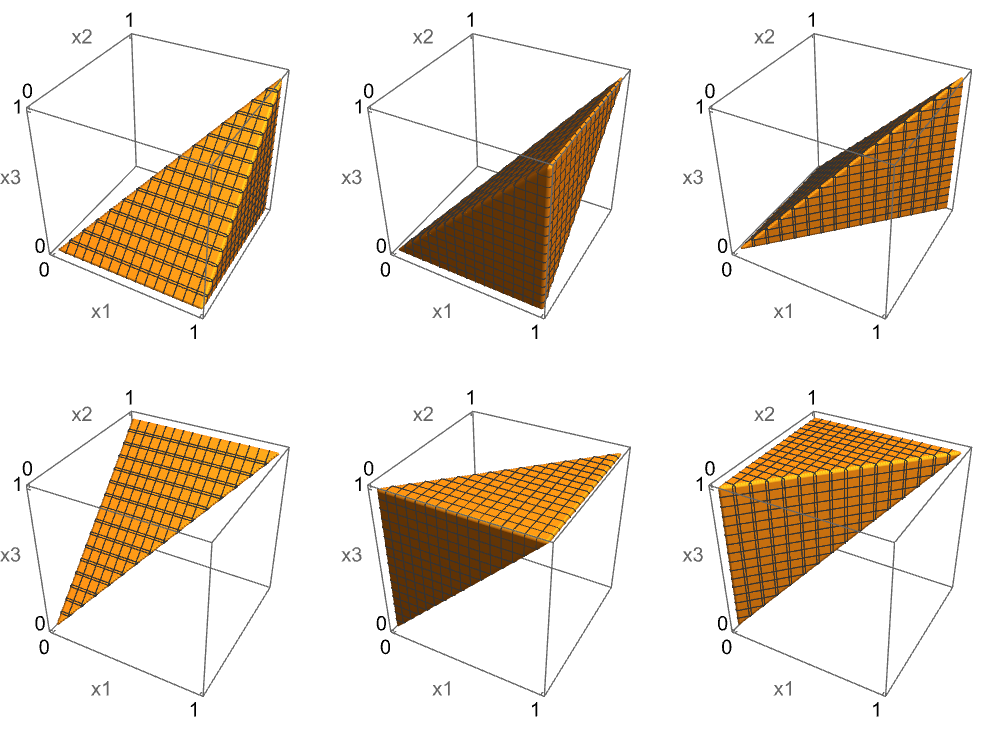}
	\caption{The six regions for $c_3$}
\end{figure}

There remains six integrals
\begin{align*}
&c_3(1)=\int_0^1\int_0^{x_1}\int_0^{x_2}(x_1-x_2)^{-d}(x_2-x_3)^{-d}(x_1-x_3)^{-d}
dx_3\, dx_2\, dx_1 \\
&c_3(2)=\int_0^1\int_0^{x_3}\int_0^{x_1}(x_3-x_1)^{-d}(x_1-x_2)^{-d}(x_3-x_2)^{-d}
dx_2\, dx_1\, dx_3 \\
&c_3(3)=\int_0^1\int_0^{x_1}\int_0^{x_3}(x_1-x_3)^{-d}(x_1-x_2)^{-d}(x_3-x_2)^{-d}
dx_2\, dx_3\, dx_1 \\
&c_3(4)=\int_0^1\int_0^{x_2}\int_0^{x_3}(x_2-x_3)^{-d}(x_2-x_1)^{-d}(x_3-x_1)^{-d}
dx_1\, dx_3\, dx_2 
\end{align*}
\begin{align*}
&c_3(5)=\int_0^1\int_0^{x_2}\int_0^{x_1}(x_2-x_1)^{-d}(x_2-x_3)^{-d}(x_1-x_3)^{-d}
dx_3\, dx_1\, dx_2 \\
&c_3(6)=\int_0^1\int_0^{x_3}\int_0^{x_2}(x_3-x_2)^{-d}(x_2-x_1)^{-d}(x_3-x_1)^{-d}
dx_1\, dx_2\, dx_3
\end{align*}
Each of these integrals can be transformed into the following integral
\[
I=\int_0^1\int_0^{y_1}\int_0^{y_2}(y_1-y_2)^{-d}(y_2-y_3)^{-d}(y_1-y_3)^{-d}
dy_3\, dy_2\, dy_1
\]
Then $c_3$ is the sum of the six integrals and we get
\begin{equation}
c_3=6\int_0^1\int_0^{y_1}\int_0^{y_2}(y_1-y_2)^{-d}(y_2-y_3)^{-d}(y_1-y_3)^{-d}
dy_3\, dy_2\, dy_1
\label{3.3}
\end{equation}
In \cite{4} the integral is computed. Here we use another method which is used extensively in the following sections. From \eqref{3.3} we get
\[
\int_0^{y_2}(y_2-y_3)^{-d}(y_1-y_3)^{-d}dy_3=
\dfrac{1}{1-d}(y_2)^{1-d}(y_1)^{-d}\hyp21{1,d}{2-d}{\dfrac{y_2}{y_1}}
\]
For the integral to $y_2$ we get
\[
I_2=\dfrac{1}{1-d}(y_1)^{-d}\int_0^{y_1}(y_1-y_2)^{-d}(y_2)^{1-d}
\hyp21{1,d}{2-d}{\dfrac{y_2}{y_1}}dy_2
\]
Writing the hypergeometric function as a summation and interchanging the integral and the summation gives
\[
I_2=\dfrac{1}{1-d}(y_1)^{-d}\sum_{k=0}^\infty\dfrac{(d)_k}{(2-d)_k}
\left(\dfrac{1}{y_1}\right)^k\int_0^{y_1}(y_1-y_2)^{-d}(y_2)^{1-d+k}dy_2
\]
The integral is well known and the result after a lot of manipulations is
\[
I_2=\dfrac{\Gamma(1-d)^2}{\Gamma(3-2d)}(y_1)^{2-3d}\hyp21{1,d}{3-2d}{1}=
\dfrac{\Gamma(1-d)^2\Gamma(2-3d)}{\Gamma(2-2d)\Gamma(3-3d)}(y_1)^{2-3d}
\]
At last we get
\[
I=\dfrac{\Gamma(1-d)^2\Gamma(2-3d)}{\Gamma(2-2d)\Gamma(3-3d)}\int_0^1(y_1)^{2-3d}dy_1=
\dfrac{\Gamma(1-d)^2\Gamma(2-3d)}{\Gamma(2-2d)\Gamma(4-3d)}
\]
Then the result for $c_3$ is
\[
c_3=6\dfrac{\Gamma(1-d)^2\Gamma(2-3d)}{\Gamma(2-2d)\Gamma(4-3d)}
\]
and this is the same as Veilette and Taqqu derived in \cite[(12)]{4}. So for the third cumulant we get
\begin{equation}
\kappa_3(d)=48\left(\dfrac{1}{2}(1-2d)(1-d)\right)^{3/2}
\dfrac{\Gamma(1-d)^2\Gamma(2-3d)}{\Gamma(2-2d)\Gamma(4-3d)}
\label{3.4}
\end{equation}
with $0 \leq d \leq 0.5$. For $d=0$ we get for the total sum of the integrals: $c_3(0)=1$. Formula \eqref{3.4} gives the following table of $\kappa_3(d)$ versus $d$.  This formula gives the same values as the values from Table 4 in \cite{5}.

\begin{table}[!ht]
	\centering
	\begin{tabular}[t]{|r|r|r|r|r|r|r|r|r|r|r|r|}
		\hline
		$d$&$0$ &0.05&0.10&0.15&0.20&0.25&0.30&0.35&0.40&0.45&0.50\\
		\hline
		$\kappa_3$ &2.828&2.815&2.770&2.684&2.548&2.348&2.067&1.686&1.183&0.5603&0\\
		\hline
	\end{tabular}
	\caption{The values of $\kappa_3(d)$ as function of $d$. $\kappa_3(d)$ is computed by formula $\eqref{3.4}$.}
	\label{Ta2}
\end{table}

In the Figure 2 we show this table.
\begin{figure}[ht]
	\centering
	\includegraphics[width=200pt]{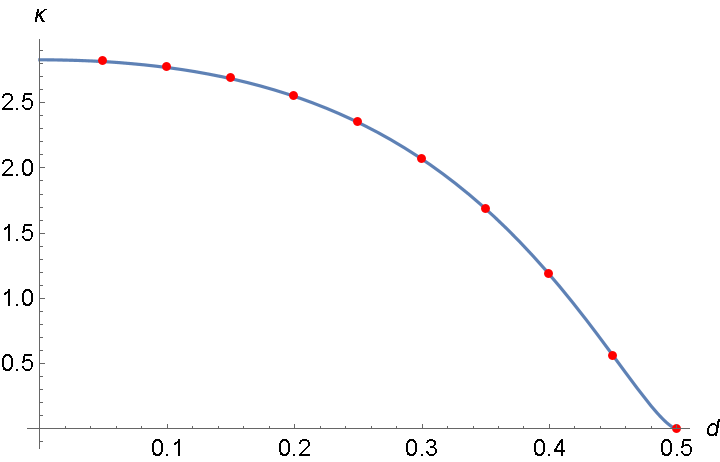}
	\caption{   The third cumulant $\kappa_3(d)$ versus $d$. The continues line is from $\eqref{3.4}$. The dots are from Table 1.}
\end{figure}
\vspace{0.1 cm}

\section{Direct derivation of the fourth cumulant}
For the fourth cumulant \eqref{3.1} gives
\begin{equation}
\kappa_4(d)=48\left(\dfrac{1}{2}(1-2d)(1-d)\right)^2c_4(d)
\label{4.1}
\end{equation}
For $c_4(d)$ we get from \eqref{3.2}
\begin{equation}
c_4(d)=\int_0^1\int_0^1\int_0^1\int_0^1 |x_1-x_2|^{-d}|x_2-x_3|^{-d}|x_3-x_4|^{-d}|x_4-x_1|^{-d}
dx_1\, dx_2\, dx_3\, dx_4
\label{4.2}
\end{equation}
There remains 24 regions and so these give 24 integrals. Transformation gives at last three different integrals and every integral occurs eight times.
\begin{align*}
&c_4(1)=\int_0^1\int_0^{y_1}\int_0^{y_2}\int_0^{y_3}
(y_1-y_2)^{-d}(y_2-y_3)^{-d}(y_3-y_4)^{-d}(y_1-y_4)^{-d}
dy_4\, dy_3\, dy_2\, dy_1 \\
&c_4(2)=\int_0^1\int_0^{y_1}\int_0^{y_2}\int_0^{y_3}
(y_1-y_2)^{-d}(y_1-y_3)^{-d}(y_2-y_4)^{-d}(y_3-y_4)^{-d}
dy_4\, dy_3\, dy_2\, dy_1 \\
&c_4(3)=\int_0^1\int_0^{y_1}\int_0^{y_2}\int_0^{y_3}
(y_1-y_3)^{-d}(y_2-y_3)^{-d}(y_1-y_4)^{-d}(y_2-y_4)^{-d}
dy_4\, dy_3\, dy_2\, dy_1
\end{align*}
Then for $c_4$ we get
\begin{equation}
c_4(d)=8\Big(c_4(1)+c_4(2)+c_4(3)\Big)
\label{4.6}
\end{equation}

\subsection{Derivation of $c_4(1)$}
For $c_4(1)$ we have
\[
c_4(1)=\int_0^1\int_0^{y_1}\int_0^{y_2}\int_0^{y_3}
(y_1-y_2)^{-d}(y_2-y_3)^{-d}(y_3-y_4)^{-d}(y_1-y_4)^{-d}
dy_4\, dy_3\, dy_2\, dy_1
\]
The integral over $y_4$ gives
\[
\int_0^{y_3}(y_3-y_4)^{-d}(y_1-y_4)^{-d}dy_4=
\dfrac{1}{(1-d)}(y_1)^{-d}(y_3)^{1-d}\hyp21{1,d}{2-d}{\dfrac{y_3}{y_1}}
\]
The integral over $y_3$ gives
\[
\int_0^{y_2}(y_3)^{1-d}(y_2-y_3)^{-d}\hyp21{1,d}{2-d}{\dfrac{y_3}{y_1}}dy_3=
\dfrac{\Gamma(1-d)\Gamma(2-d)}{\Gamma(3-2d)}(y_2)^{2-2d}\hyp21{1,d}{3-2d}{\dfrac{y_2}{y_1}}
\]
Integration over $y_2$ and $y_1$ gives
\begin{multline*}
c_4(1)=\dfrac{\Gamma(1-d)\Gamma(2-d)}{\Gamma(3-2d)}\int_0^1\int_0^{y_1}(y_1-y_2)^{-d}
(y_2)^{2-2d}\hyp21{1,d}{3-2d}{\dfrac{y_2}{y_1}}dy_2\, dy_1= \\
=\dfrac{\Gamma(1-d)^3\Gamma(3-4d)}{\Gamma(3-3d)\Gamma(5-4d)}
 \end{multline*}

\subsection{Derivation of $c_4(2)$}
For $c_4(2)$ we have
\[
c_4(2)=\int_0^1\int_0^{y_1}\int_0^{y_2}\int_0^{y_3}
(y_1-y_2)^{-d}(y_1-y_3)^{-d}(y_2-y_4)^{-d}(y_3-y_4)^{-d}
dy_4\, dy_3\, dy_2\, dy_1
\]
The integral over $y_4$ gives
\[
\int_0^{y_3}(y_2-y_4)^{-d}(y_3-y_4)^{-d}dy_4=
\dfrac{1}{(1-d)}(y_2)^{-d}(y_3)^{{1-d}}\hyp21{1,d}{2-d}{\dfrac{y_3}{y_1}}
\]
Writing the hypergeometric function as a summation and integration over $y_3$ gives
\[
c_4(2)=\dfrac{1}{(1-d)}\dfrac{\Gamma(1-d)}{\Gamma(3-d)}\sum_{k=0}^\infty
\int_0^1\int_0^{y_1}(y_1)^{-d}(y_2)^{2-2d}(y_1-y_2)^{-d}
\hyp21{d,2-d+k}{3-d+k}{\dfrac{y_2}{y_1}}dy_2\, dy_1
\]
Integration over $y_2$ and $y_1$ gives
\[
c_4(2)=\dfrac{\Gamma(3-2d)\Gamma(1-d)^2}{(4-4d)\Gamma(4-3d)\Gamma(3-d)}
\sum_{k=0}^\infty\dfrac{(d)_k}{(3-d)_k}\hyp32{2-d+k,d,3-2d}{3-d+k,4-3d}{1}
\]
After using the Thomae transformation \eqref{55.1} we get
\[
c_4(2)=2\dfrac{\Gamma(1-d)}{\Gamma(3-d)}
\dfrac{\Gamma(2-2d)^2}{\Gamma(5-4d)}
\sum_{k=0}^\infty\dfrac{(d)_k}{(3-d)_k}
\hyp32{1,d,d+k}{2-d,3-d+k}{1}
\]
After writing the hypergeometric function as a summation over $j$, do the summation over $k$. Then the summation over $j$ gives
\[
c_4(2)=\dfrac{\Gamma(2-2d)^2}{(1-d)^2\Gamma(5-4d|)}\hyp32{1,d,d}{2-d,2-d}{1}
\]
Using the Thomae transformation \eqref{55.5} gives after a lot of manipulations
\begin{equation*}
c_4(2)=\dfrac{\Gamma(1-d)^4\Gamma(3-4d)}{2\Gamma(2-2d)^2\Gamma(5-4d)}
\end{equation*}

\subsection{Derivation of $c_4(3)$}
For $c_4(3)$ we have
\[
c_4(3)=\int_0^1\int_0^{y_1}\int_0^{y_2}\int_0^{y_3}
(y_1-y_3)^{-d}(y_2-y_3)^{-d}(y_1-y_4)^{-d}(y_2-y_4)^{-d}
dy_4\, dy_3\, dy_2\, dy_1
\]
In this section we describe two methods to compute this integral. The reason for this is that both methods are frequently used for the derivation of the fifth cumulant of the Rosenblatt distribution.

For the first method we use the property of Remark 4. Application gives
\begin{multline*}
\int_0^{y_2}\int_0^{y_3}
(y_1-y_3)^{-d}(y_2-y_3)^{-d}(y_1-y_4)^{-d}(y_2-y_4)^{-d}dy_4\, dy_3= \\
=\int_0^{y_2}(y_1-y_4)^{-d}(y_2-y_4)^{-d}
\int_{y_4}^{y_2}(y_1-y_3)^{-d}(y_2-y_3)^{-d}dy_3\, dy_4
\end{multline*}
Using the transformation $y_3=\big(y_2-y_4\big)z+y_4$ gives
\begin{multline*}
\int_0^{y_2}(y_1-y_4)^{-2d}(y_2-y_4)^{1-2d}
\int_0^1 (1-z)^{-d}\left(1-\dfrac{y_2-y_4}{y_1-y_4}z\right)^{-d}dz\, dy_1= \\
=\dfrac{1}{1-d}\int_0^{y_2}(y_1-y_4)^{-2d}(y_2-y_4)^{1-2d}
\hyp21{1,d}{2-d}{\dfrac{y_2-y_4}{y_1-y_4}} dy_1
\end{multline*}
Writing the hypergeometric function as a summation and interchanging the integral and the summation gives
\begin{multline*}
\dfrac{1}{1-d}\sum_{k=0}^\infty\dfrac{(d)_k}{(2-d)_k}
\int_0^{y_2}(y_1-y_4)^{-2d-k}(y_2-y_4)^{1-2d+k}dy_4= \\
=\dfrac{1}{1-d}\dfrac{\Gamma(2-2d)}{\Gamma(3-2d)}
\sum_{k=0}^\infty\dfrac{(d)_k(2-2d)_k}{(2-d)_k(3-2d)_k}
(y_1)^{-2d-k}(y_2)^{2-2d+k}\hyp21{1,2d+k}{3-2d+k}{\dfrac{y_2}{y_1}}
\end{multline*}
Substitution in $c_4(3)$ gives
\begin{multline*}
c_4(3)=\dfrac{1}{1-d}\dfrac{\Gamma(2-2d)}{\Gamma(3-2d)}
\sum_{k=0}^\infty\dfrac{(d)_k(2-2d)_k}{(2-d)_k(3-2d)_k} \\
\int_0^1\int_0^{y_1}(y_1)^{-2d-k}(y_2)^{2-2d+k}\hyp21{1,2d+k}{3-2d+k}{\dfrac{y_2}{y_1}}
dy_2\, dy_1
\end{multline*}
For the inner integral we get after using the transformation $y_2=y_1\, z$
\begin{multline*}
\int_0^{y_1}(y_2)^{2-2d+k}\hyp21{1,2d+k}{3-2d+k}{\dfrac{y_2}{y_1}}dy_2= \\
=(y_1)^{3-2d+k}\int_0^1\, z^{2-2d+k}\hyp21{1,2d+k}{3-2d+k}{z}dz=
\dfrac{\Gamma(3-4d)}{\Gamma(4-4d)}(y_1)^{3-2d+k}
\end{multline*}
Substitution gives
\begin{align*}
c_4(3)&=\dfrac{1}{1-d}\dfrac{\Gamma(2-2d)}{\Gamma(3-2d)}\dfrac{\Gamma(3-4d)}{\Gamma(4-4d)}
\sum_{k=0}^\infty\dfrac{(d)_k(2-2d)_k}{(2-d)_k(3-2d)_k}
\int_0^1(y_1)^{3-4d}dy_1= \\
&=\dfrac{1}{1-d}\dfrac{\Gamma(2-2d)}{\Gamma(3-2d)}\dfrac{\Gamma(3-4d)}{\Gamma(5-4d)}
\sum_{k=0}^\infty\dfrac{(d)_k(2-2d)_k}{(2-d)_k(3-2d)_k}
\end{align*}
The summation can be written as a $_3F_2$ hypergeometric function. After some manipulations we get at last
\begin{equation}
c_4(3)=\dfrac{1}{2(1-d)^2}\dfrac{\Gamma(3-4d)}{\Gamma(5-4d)}
\hyp32{1,d,2-2d}{2-d,3-2d}{1}
\label{4.9}
\end{equation}

\

\

For the second method we use the transformations of Remark 5 and get
\[
c_4(3)=\dfrac{1}{(4-4d)}\int_0^1\int_0^1\int_0^1
(z_2)^{2-2d}(z_3)(1-z_3)^{-d}(1-z_2\, z_3)^{-d}(1-z_2\, z_3\, z_4)^{-d}(1-z_3\, z_4)^{-d}
dz_2\, dz_3\, dz_4
\]
Writing the factor $(1-z_2\, z_3\, z_4)^{-d}$ as a summation gives after integration over $z_2$ 
\begin{multline*}
c_4(3)=\dfrac{1}{(4-4d)}\sum_{k=0}^\infty\dfrac{(d)_k}{(3-2d+k)k!} \\
\int_0^1\int_0^1(1-z_3)^{-d})z_3)^{1+k}(z_4)^k(1-z_3\, z_4)^{-d}
\hyp21{d,3-2d+k}{4-2d+k}{z_3} dz_3\, dz_4
\end{multline*}
Integration over $z_4$ gives
\begin{multline*}
c_4(3)=\dfrac{1}{(4-4d)}\sum_{k=0}^\infty\dfrac{(d)_k}{(3-2d+k))(1+k)k!} \\
\int_0^1(z_3)^{(1-d)-1}(1-z_3)^{1+k}
\hyp21{d,1+k}{2+k}{1-z_3}\hyp21{d,3-2d+k}{4-2d+k}{1-z_3}dz_3
\end{multline*}
From \cite[2.21.9(12)]{2} there is
\begin{align*}
&\int_0^1x^{\alpha-1}(1-x)^{c-1}\hyp21{a,b}{c}{1-x}\hyp21{a',b'}{c'}{1-x}dx=  \\
&=\dfrac{\Gamma(c)\Gamma(c')\Gamma(c'-a'-b')\Gamma)(\alpha)\Gamma(c-a-b+\alpha)}
{\Gamma(c-a+\alpha)\Gamma(c-b+\alpha)\Gamma(c'-a')\Gamma(c'-b')}
\hyp43{a',b',\alpha,c-a-b+\alpha}{c-a+\alpha,c-b+\alpha,a'+b'-c'+1}{1}+ \\
&+\dfrac{\Gamma(c)\Gamma(c')\Gamma(a'+b'-\underline{c'})\Gamma(c'-a'-b'+\alpha)\Gamma(c+c'-a-a'-b-b'+\alpha)}{\Gamma(a')\Gamma(b')\Gamma(c+c'-a-a'-b'+\alpha)\Gamma(c+c'-a'-b-b'+\alpha)} \\
&\qquad\qquad\qquad
\hyp43{c'-a',c'-b',c'-a'-b'+\alpha,c+c'-a-a'-b-b'+\alpha}{c'-a'-b'+1,c+c'-a-a'-b'+\alpha,c+c'-a'-b-b'+\alpha}{1}
\end{align*}
Note the little mistake with the underlined $c'$ which is here corrected. Application gives after summation over $k$
\[
c_4(3)=\dfrac{\Gamma(4-3d)\Gamma(1-d)^2\Gamma(2-2d)\Gamma(2d-1)}
{\Gamma(5-4d)\Gamma(3-3d)\Gamma(d)}-
\dfrac{\Gamma(3-4d)\Gamma(2d-1)}{(1-d)\Gamma(5-4d)\Gamma(2d)}
\hyp32{2d-1,d,1}{2d,2-d}{1}
\]
Using the Thomae transformation \eqref{55.3} gives the same result as \eqref{4.9}.

\subsection{The fourth cumulant}
For $c_4$ the summation of $8\big(c_4(1)+c_4(2)+c_4(3)\big)$ gives
\[
c_4=
\dfrac{1}{(1-d)^2(3-4d)}\left(\hyp32{2-2d,1,d}{3-2d,2-d}{1}+\dfrac{\Gamma(1-d)^2\Gamma(2-d)^2}{\Gamma(2-2d)^2}+
\dfrac{2\Gamma(1-d)\Gamma(2-d)^2}{\Gamma(3-3d)}\right)
\]
For the fourth cumulant we use \eqref{4.1} and \eqref{4.6}. The result is
\[
\kappa_4(d)=12(1-2d)^2(1-d)^2c_4 
\]
Substitution gives for the fourth cumulant
\begin{multline}
\kappa_4(d)=\dfrac{12(1-2d)^2}{(1-d)(3-4d)} \\
\left(\hyp32{2-2d,1,d}{3-2d,2-d}{1}+\dfrac{\Gamma(1-d)^2\Gamma(2-d)^2}{\Gamma(2-2d)^2}+
\dfrac{2\Gamma(1-d)\Gamma(2-d)^2}{\Gamma(3-3d)}\right)
\label{4.10}
\end{multline}
with $0 \leq d \leq 0.5$. For $d=0$ we get for the total sum of the integrals: $8\, \big(c_4(1)+c_4(2)+c_4(3)\big)=1$. Formula \eqref{4.10} gives the following table of $\kappa_4$ versus $d$. This formula gives the same values as the values from Table 4 in \cite{5}.

\begin{table}[!ht]
\centering
\begin{tabular}[t]{|r|r|r|r|r|r|r|r|r|r|r|r|}
	\hline
	$d$&$0$ &0.05&0.10&0.15&0.20&0.25&0.30&0.35&0.40&0.45&0.50\\
	\hline
	$\kappa_4$ &12.00&11.92&11.66&11.15&10.35&9.192&7.632&5.665&3.392&1.173&0\\
	\hline
\end{tabular}
\caption{The values of $\kappa(d)$ as function of $d$. $\kappa_4(d)$ is computed by formula $\eqref{4.10}$.}
\label{Ta2}
\end{table}

In the next figure this table is shown.
\begin{figure}[ht]
	\centering
	\includegraphics[width=200pt]{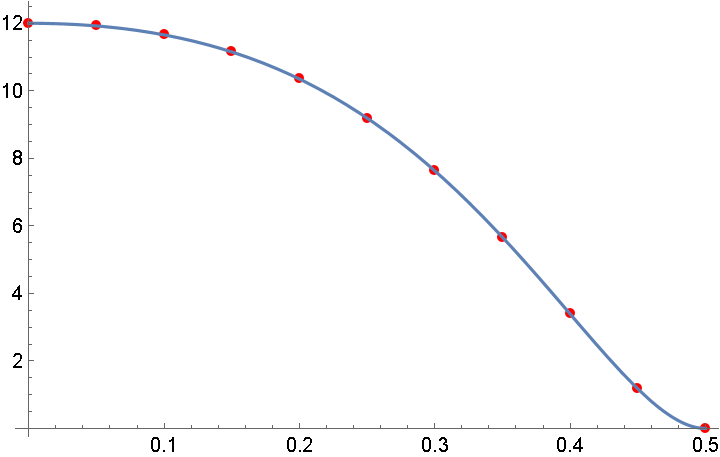}
	\caption{   The fourth cumulant $\kappa_4(d)$ versus $d$. The continues line is from $\eqref{4.10}$. The dots are from Table 2.}
\end{figure}
\vspace{0.1 cm}

\section{Direct derivation of the fifth cumulant}
For the fifth cumulant \eqref{1.3} gives
\begin{equation}
\kappa_5(d)=384\left(\dfrac{1}{2}(1-2d)(1-d)\right)^{5/2} c_5
\label{53.1}
\end{equation}
From this formula follows immediately that $0 \leq d \leq 0.5$. For $c_5$ we get from \eqref{1.2}
\begin{multline}
c_5=\int_0^1\int_0^1\int_0^1\int_0^1\int_0^1 |x_1-x_2|^{-d}|x_2-x_3|^{-d}|x_3-x_4|^{-d}(x_4-x_5)^{-d} \\
|x_5-x_1|^{-d}dx_1\, dx_2\, dx_3\, dx_4\, dx_5
\label{54.2}
\end{multline}
The factors with the absolute values should be split up. There remains 120 regions and so these give 120 integrals. Using the transformations of Remark 5 there remains twelve different integrals and every integral occurs ten times.
The regions are
\begin{align*}
&(\ 1)\qquad 0<x_5<x_4<x_3<x_2<x_1<1 \\
&(\ 2)\qquad 0<x_4<x_5<x_3<x_2<x_1<1 \\
&(\ 3)\qquad0<x_4<x_3<x_5<x_2<x_1<1 \\
&(\ 4)\qquad0<x_5<x_3<x_4<x_2<x_1<1 \\
&(\ 5)\qquad0<x_3<x_5<x_4<x_2<x_1<1 \\
&(\ 6)\qquad0<x_5<x_3<x_2<x_4<x_1<1 \\
&(\ 7)\qquad0<x_3<x_5<x_2<x_4<x_1<1 \\
&(\ 8)\qquad0<x_5<x_3<x_2<x_1<x_4<1 \\
&(\ 9)\qquad0<x_4<x_3<x_2<x_5<x_1<1 \\
&(10)\qquad0<x_3<x_2<x_5<x_4<x_1<1 \\
&(11)\qquad0<x_3<x_5<x_2<x_1<x_4<1 \\
&(12)\qquad0<x_3<x_2<x_5<x_1<x_4<1
\end{align*}
As an example how to get the transformed integrals we take region (3). The integral becomes
\begin{multline*}
c_5(3)=\int_0^1\int_0^{x_1}(x_1-x_2)^{-d}\int_0^{x_2}(x_1-x_5)^{-d}
\int_0^{x_5}(x_2-x_5)^{-d} \\
\qquad\qquad\qquad\qquad\qquad\qquad\qquad\qquad\qquad
\int_0^{x_3}(x_5-x_4)^{-d}(x_3-x_4)^{-d}
dx_4\, dx_3\, dx_5\, dx_2\, dx_1
\end{multline*}
Transforming  this integral gives
\begin{multline*}
c_5(3)=\int_0^1\int_0^{y_1}(y_1-y_2)^{-d}\int_0^{y_2}(y_1-y_3)^{-d}
\int_0^{y_3}(y_2-y_4)^{-d} \\
\qquad\qquad\qquad\qquad\qquad\qquad\qquad\qquad\qquad
\int_0^{y_4}(y_3-y_5)^{-d}(y_4-y_5)^{-d}
dy_5\, dy_4\, dy_3\,  dy_2\, dy_1	
\end{multline*}	
When applying the transformations of Remark 5 we get
\begin{multline*}
c_5(3)=\dfrac{1}{(5-5d)}\int_0^1\int_0^1\int_0^1\int_0^1
(z_2)^{3-3d}(1-z_2)^{-d}(z_3)^{2-2d}(1-z_2\, z_3)^{-d} \\
(z_4)^{1-d}(1-z_3\, z_4)^{-d}(1-z_4\, z_5)^{-d}(1-z_5)^{-d}dz_5\, dz_4\, dz_3\, dz_2
\end{multline*}
These transformations can be applied for all 120 integrals and we can compare all these integrals. There remains the following integrals
\begin{align*}
&c_5(1)=\int_0^1\int_0^{y_1}(y_1-y_2)^{-d}\int_0^{y_2}(y_2-y_3)^{-d}
\int_0^{y_3}(y_3-y_4)^{-d} \\
&\qquad\qquad\qquad\qquad\qquad\qquad\qquad\qquad\qquad \int_0^{y_4}(y_4-y_5)^{-d}(y_1-y_5)^{-d}
dy_5\, dy_4\, dy_3\,  dy_2\, dy_1 \\
&c_5(2)=\int_0^1\int_0^{y_1}(y_1-y_2)^{-d}\int_0^{y_2}(y_2-y_3)^{-d}
\int_0^{y_3}(y_1-y_4)^{-d} \\
&\qquad\qquad\qquad\qquad\qquad\qquad\qquad\qquad\qquad
\int_0^{y_4}(y_4-y_5)^{-d}(y_3-y_5)^{-d}
dy_5\, dy_4\, dy_3\,  dy_2\, dy_1 
\end{align*}
\begin{align*}
&c_5(3)=\int_0^1\int_0^{y_1}(y_1-y_2)^{-d}\int_0^{y_2}(y_1-y_3)^{-d}
\int_0^{y_3}(y_2-y_4)^{-d} \\
&\qquad\qquad\qquad\qquad\qquad\qquad\qquad\qquad\qquad
\int_0^{y_4}(y_3-y_5)^{-d}(y_4-y_5)^{-d}
dy_5\, dy_4\, dy_3\,  dy_2\, dy_1 \\
&c_5(4)=\int_0^1\int_0^{y_1}(y_1-y_2)^{-d}\int_0^{y_2}\int_0^{y_3}
(y_2-y_4)^{-d}(y_3-y_4)^{-d} \\
&\qquad\qquad\qquad\qquad\qquad\qquad\qquad\qquad\qquad
\int_0^{y_4}(y_3-y_5)^{-d}(y_1-y_5)^{-d}
dy_5\, dy_4\, dy_3\,  dy_2\, dy_1 \\
&c_5(5)=\int_0^1\int_0^{y_1}(y_1-y_2)^{-d}\int_0^{y_2}\int_0^{y_3}
(y_3-y_4)^{-d}(y_1-y_4)^{-d} \\
&\qquad\qquad\qquad\qquad\qquad\qquad\qquad\qquad\qquad
\int_0^{y_4}(y_2-y_5)^{-d}(y_3-y_5)^{-d}
dy_5\, dy_4\, dy_3\,  dy_2\, dy_1 \\
&c_5(6)=\int_0^1\int_0^{y_1}\int_0^{y_2}(y_1-y_3)^{-d}\int_0^{y_3}
(y_3-y_4)^{-d}(y_2-y_4)^{-d} \\
&\qquad\qquad\qquad\qquad\qquad\qquad\qquad\qquad\qquad
\int_0^{y_4}(y_2-y_5)^{-d}(y_1-y_5)^{-d}
dy_5\, dy_4\, dy_3\,  dy_2\, dy_1 \\
&c_5(7)=\int_0^1\int_0^{y_1}\int_0^{y_2}(y_1-y_3)^{-d}\int_0^{y_3}
(y_2-y_4)^{-d}(y_1-y_4)^{-d} \\
&\qquad\qquad\qquad\qquad\qquad\qquad\qquad\qquad\qquad
\int_0^{y_4}(y_2-y_5)^{-d}(y_3-y_5)^{-d}
dy_5\, dy_4\, dy_3\,  dy_2\, dy_1 \\
&c_5(8)=\int_0^1\int_0^{y_1}\int_0^{y_2}(y_2-y_3)^{-d}\int_0^{y_3}
(y_3-y_4)^{-d}(y_1-y_4)^{-d} \\
&\qquad\qquad\qquad\qquad\qquad\qquad\qquad\qquad\qquad
\int_0^{y_4}(y_1-y_5)^{-d}(y_2-y_5)^{-d}dy_5\, dy_4\, dy_3\,  dy_2\, dy_1 \\
&c_5(9)=\int_0^1\int_0^{y_1}(y_1-y_2)^{-d}\int_0^{y_2}(y_1-y_3)^{-d}\int_0^{y_3}
(y_3-y_4)^{-d} \\
&\qquad\qquad\qquad\qquad\qquad\qquad\qquad\qquad\qquad
\int_0^{y_4}(y_2-y_5)^{-d}(y_4-y_5)^{-d}dy_5\, dy_4\, dy_3\,  dy_2\, dy_1 \\
&c_5(10)=\int_0^1\int_0^{y_1}\int_0^{y_2}(y_1-y_3)^{-d}(y_2-y_3)^{-d|}\int_0^{y_3}
(y_1-y_4)^{-d} \\
&\qquad\qquad\qquad\qquad\qquad\qquad\qquad\qquad\qquad
\int_0^{y_4}(y_2-y_5)^{-d}(y_4-y_5)^{-d}dy_5\, dy_4\, dy_3\,  dy_2\, dy_1 \\
&c_5(11)=\int_0^1\int_0^{y_1}\int_0^{y_2}(y_2-y_3)^{-d}\int_0^{y_3}
(y_1-y_4)^{-d}(y_2-y_4)^{-d} \\
&\qquad\qquad\qquad\qquad\qquad\qquad\qquad\qquad\qquad
\int_0^{y_4}(y_1-y_5)^{-d}(y_3-y_5)^{-d}dy_5\, dy_4\, dy_3\,  dy_2\, dy_1 \\
&c_5(12)=\int_0^1\int_0^{y_1}\int_0^{y_2}(y_1-y_3)^{-d}(y_2-y_3)^{-d}\int_0^{y_3}
(y_2-y_4)^{-d}\\
&\qquad\qquad\qquad\qquad\qquad\qquad\qquad\qquad\qquad
\int_0^{y_4}(y_1-y_5)^{-d}(y_4-y_5)^{-d}dy_5\, dy_4\, dy_3\,  dy_2\, dy_1 
\end{align*}
and every integral occurs ten times. Then for the fifth cumulant $c_5$ we get
\begin{equation}
c_5=10\sum_{i=1}^{12}c_5(i)
\label{53.3}
\end{equation}

\subsection{Derivation of $c_5(1)$}
For $c_5(1)$ we have
\begin{multline*}
c_5(1)=\int_0^1\int_0^{y_1}(y_1-y_2)^{-d}\int_0^{y_2}(y_2-y_3)^{-d}
\int_0^{y_3}(y_3-y_4)^{-d} \\
\int_0^{y_4}(y_4-y_5)^{-d}(y_1-y_5)^{-d}dy_5\, dy_4\, dy_3\,  dy_2\, dy_1	
\end{multline*}
The integral for $y_5$ gives
\[
\int_0^{y_4}(y_4-y_5)^{-d}(y_1-y_5)^{-d}dy_5=
\dfrac{1}{(1-d)}(y_1)^{-d}(y_4)^{1-d}\hyp21{1,d}{2-d}{\dfrac{y_4}{y_1}}
\]
The integral over $y_4$ gives
\begin{multline*}
\dfrac{1}{(1-d)}(y_1)^{-d}\int_0^{y_3}(y_3-y_4)^{-d}(y_4)^{1-d}
\hyp21{1,d}{2-d}{\dfrac{y_4}{y_1}}dy_4= \\
=\dfrac{\Gamma(1-d)^2}{\Gamma(3-2d)}(y_1)^{-d}(y_3)^{2-2d}
\hyp21{1,d}{3-2d}{\dfrac{y_3}{y_1}}
\end{multline*}

The integral over $y_3$ gives
\begin{multline*}
\dfrac{\Gamma(1-d)^2}{\Gamma(3-2d)}(y_1)^{-d}\int_0^{y_2}(y_2-y_3)^{-d}(y_3)^{2-d}
\hyp21{1,d}{3-2d}{\dfrac{y_3}{y_1}}dy_3= \\
=\dfrac{\Gamma(1-d)^3}{\Gamma(4-3d)}(y_1)^{-d}(y_2)^{3-3d}
\hyp21{1,d}{4-3d}{\dfrac{y_2}{y_1}}
\end{multline*}
The integral over $y_2$ gives
\begin{multline*}
\dfrac{\Gamma(1-d)^3}{\Gamma(4-3d)}(y_1)^{-d}
\int_0^{y_1}(y_1-y_2)^{-d}(y_2)^{3-3d}\hyp21{1,d}{4-3d}{\dfrac{y_2}{y_1}}dy_2= \\
=\dfrac{\Gamma(1-d)^4}{\Gamma(5-4d)}(y_1)^{4-5d}\hyp21{1,d}{5-4d}{1}
\end{multline*}
The integral over $y_1$ gives
\[
\dfrac{\Gamma(1-d)^4\Gamma(4-5d)}{\Gamma(4-4d)\Gamma(5-5d)}\int_0^1(y_1)^{4-5d}dy_1=
\dfrac{\Gamma(1-d)^4\Gamma(4-5d)}{\Gamma(4-4d)\Gamma(5-5d)}\dfrac{1}{(5-5d)}
\]
Then for $c_5(1)$ we get at last
\[
c_5(1)=\dfrac{\Gamma(1-d)^4\Gamma(4-5d)}{\Gamma(4-4d)\Gamma(6-5d)}
\]

\subsection{Derivation of $c_5(2)$}
For $c_5(2)$ we have
\begin{multline*}
	c_5(2)=\int_0^1\int_0^{y_1}(y_1-y_2)^{-d}\int_0^{y_2}(y_2-y_3)^{-d}
	\int_0^{y_3}(y_1-y_4)^{-d} \\
	\qquad\qquad\qquad\qquad\qquad\qquad\qquad\qquad\qquad
	\int_0^{y_4}(y_4-y_5)^{-d}(y_3-y_5)^{-d}dy_5\, dy_4\, dy_3\,  dy_2\, dy_1
\end{multline*}
The integral for $y_5$ gives
\[
\int_0^{y_4}(y_4-y_5)^{-d}(y_3-y_5)^{-d}dy_5=
\dfrac{1}{(1-d)}(y_3)^{-d}(y_4)^{1-d}\hyp21{1,d}{2-d}{\dfrac{y_4}{y_3}}
\]
Writing the hypergeometric function as a summation and interchange the summation and the integral gives
\begin{multline*}
\dfrac{1}{(1-d)}\int_0^{y_3}(y_3)^{-d}(y_4)^{1-d}
\hyp21{1,d}{2-d}{\dfrac{y_4}{y_3}}dy_4= \\
=\dfrac{1}{(1-d)}\dfrac{\Gamma(2-d)}{\Gamma(3-d)}(y_1)^{-d}
\sum_{k=0}^\infty\dfrac{(d)_k}{(3-d)_k}(y_3)^{2-2d}
\hyp21{d,2-d+k}{3-d+k}{\dfrac{y_3}{y_1}}
\end{multline*}
Writing the hypergeometric function as a summation and interchange the summation and the integral gives
\begin{multline*}
\dfrac{\Gamma(1-d)}{\Gamma(3-d)}
\sum_{k=0}^\infty\dfrac{(d)_k}{(3-d)_k}
\sum_{j=0}^\infty\dfrac{(d)_j(2-d+k)_j}{(3-d+k)_j}\dfrac{1}{j!}(y_1)^{-j-d}
\int_0^{y_2}(y_3)^{2-2d+j}(y_2-y_3)^{-d}dy_3= \\
=\dfrac{\Gamma(1-d)^2}{\Gamma(3-d)}
\sum_{k=0}^\infty\dfrac{(d)_k}{(3-d)_k}
\sum_{j=0}^\infty\dfrac{(d)_j(2-d+k)_j}{(3-d+k)_j}\dfrac{1}{j!}(y_1)^{-j-d}
\dfrac{\Gamma(3-2d+j)}{\Gamma(4-3d+j)}(y_2)^{3-3d+j}
\end{multline*}
The integral over $y_2$ gives
\begin{align}
&\dfrac{\Gamma(1-d)^2}{\Gamma(3-d)}
\sum_{k=0}^\infty\dfrac{(d)_k}{(3-d)_k}
\sum_{j=0}^\infty\dfrac{(d)_j(2-d+k)_j}{(3-d+k)_j}\dfrac{1}{j!}(y_1)^{-j-d}
\dfrac{\Gamma(3-2d+j)}{\Gamma(4-3d+j)} \nonumber \\
&\qquad\qquad\qquad\qquad\qquad\qquad\qquad\qquad\qquad\qquad\qquad\qquad
\int_0^{y_1}(y_2)^{3-3d+j}(y_1-y_2)^{-d}dy_2= \nonumber \\
&=\dfrac{\Gamma(1-d)^2}{\Gamma(3-d)}\dfrac{\Gamma(3-2d)}{\Gamma(5-4d)}
\sum_{k=0}^\infty\dfrac{(d)_k}{(3-d)_k}
\sum_{j=0}^\infty\dfrac{(d)_j(3-2d)_j(2-d+k)_j}{(5-4d)_j(3-d+k)_j}
\dfrac{1}{j!}(y_1)^{4-5d}
\label{53.2.1}
\end{align}

The next steps should be done 
\begin{description}
\item[-] Integrate \eqref{53.2.1} over $y_1$.
\item[-] Apply Thomae transformation \eqref{55.1}.
\item[-] Interchange summations and do summation over $k$.
\item[-] Apply Thomae transformation \eqref{55.1}.
\item[-] Write the summation as a hypergeometric function.
\end{description}

The result is
\[
c_5(2)=\dfrac{\Gamma(1-d)^2}{(1-d)}
\dfrac{\Gamma(4-5d)\Gamma(2-2d)}{\Gamma(6-5d)\Gamma(4-4d)}
\hyp32{d,1-d,2-2d}{2-d,4-4d}{1}
\]

\subsection{Derivation of $c_5(3)$}
For $c_5(3)$ we have
\begin{multline*}
c_5(3)=\int_0^1\int_0^{y_1}(y_1-y_2)^{-d}\int_0^{y_2}(y_1-y_3)^{-d}
\int_0^{y_3}(y_2-y_4)^{-d} \\
\int_0^{y_4}(y_3-y_5)^{-d}(y_4-y_5)^{-d}dy_5\, dy_4\, dy_3\,  dy_2\, dy_1
\end{multline*}
Using \eqref{1.1} and using the transformation $y_4=(y_3-y_5)z+y_5$ gives
\begin{multline*}
c_5(3)=\int_0^1\int_0^{y_1}(y_1-y_2)^{-d}\int_0^{y_2}(y_1-y_3)^{-d}
\int_0^{y_3}(y_3-y_5)^{-d} \\
\int_0^1 z^{-d}\left(1-\dfrac{y_3-y_5}{y_2-y_5}z\right)^{-d}dz\, dy_5\, dy_3\,  dy_2\, dy_1
\end{multline*}
The integral over $z$ gives
\[
\int_0^1z^{-d}\left(1-\dfrac{y_{3}-y_{5}}{y_{2}-y_{5}}z\right)^{-d}dz=
\dfrac{1}{(1-d)}(y_{3}-y_{5})^{1-d}(y_{2}-y_{5}) ^{-d}
\hyp21{1-d,d}{2-d}{\dfrac{y_{3}-y_{5}}{y_{2}-y_{5}}}
\]
After writing the hypergeometric function as a summation the integral over $y_5$ gives
\begin{multline*}
\dfrac{1}{(1-d)}\int_0^{y_{3}}(y_{3}-y_{5})^{1-2d}( y_{2}-y_{5})^{-d}
\hyp21{1-d,d}{2-d}{\dfrac{y_{3}-y_{5}}{y_{2}-y_{5}}}dy_5= \\
=\dfrac{1}{(1-d)}\dfrac{\Gamma(2-2d)}{\Gamma(3-2d)}
\sum_{k=0}^\infty\dfrac{(1-d)_{k}(d)_{k}}{(2-d)_{k}}\dfrac{(2-2d)_{k}}{(3-2d)_{k}}
\dfrac{1}{k!}(y_{2}) ^{-d-k}(y_{3})^{2-2d+k}
\sum_{j=0}^\infty\dfrac{(d+k)_j}{(3-2d+k)_j}\left(\dfrac{y_{3}}{y_{2}}\right)^j
\end{multline*}
The integral over $y_3$ gives
\begin{align*}
&\dfrac{1}{(1-d)}\dfrac{\Gamma(2-2d)}{\Gamma(3-2d)}
\sum_{k=0}^\infty\dfrac{(1-d)_{k}(d)_{k}}{(2-d)_{k}}\dfrac{(2-2d)_{k}}{(3-2d)_{k}}
\dfrac{1}{k!}
\sum_{j=0}^\infty\dfrac{(d+k)_j}{(3-2d+k)_j}(y_{2}) ^{-d-k-j} \\
&\qquad\qquad\qquad\qquad\qquad\qquad\qquad\qquad\qquad\qquad\qquad\qquad\qquad
\int_0^{y_2}(y_{3})^{2-2d+k+j}(y_1-y_3)^{-d}dy_3= \\
&=\dfrac{1}{(1-d)}\dfrac{\Gamma(2-2d)}{\Gamma(3-2d)}
\sum_{k=0}^\infty\dfrac{(1-d)_{k}(d)_{k}}{(2-d)_{k}}\dfrac{(2-2d)_{k}}{(3-2d)_{k}}
\dfrac{1}{k!}
\sum_{j=0}^\infty\dfrac{(d+k)_j}{(3-2d+k)_j}(y_{2}) ^{-d-k-j} \\
&\qquad\qquad\qquad\qquad\qquad\qquad\qquad\qquad
\dfrac{1}{(3-2d+j+k)}(y_1)^{-d}(y_2)^{3-3d}\hyp21{d,3-2d+j+k}{4-2d+j+k}{\dfrac{y_2}{y_1}}
\end{align*}
The integral over $y_2$ gives
\begin{align*}
&\dfrac{1}{(1-d)}\dfrac{\Gamma(2-2d)}{\Gamma(4-2d)}
\sum_{k=0}^\infty\dfrac{(1-d)_k(d)_k}{(2-d)_k}\dfrac{(2-2d)_k}{(4-2d)_k}\dfrac{1}{k!}
\sum_{j=0}^\infty\dfrac{(d+k)_j}{(4-2d+k)_j}(y_{1})^{-d} \\
&\qquad\qquad\qquad\qquad\qquad\qquad\qquad\qquad
\int_0^{y_{1}}(y_{2})^{3-3d}(y_{1}-y_{2})^{-d}
\hyp21{d,3-2d+j+k}{4-2d+j+k}{\dfrac{y_2}{y_1}}dy_{2}= \\
&=\dfrac{\Gamma(1-d)}{(1-d)}\dfrac{\Gamma(2-2d)}{\Gamma(4-2d)}
\dfrac{\Gamma(4-3d)}{\Gamma(5-4d)}
\sum_{k=0}^\infty\dfrac{(1-d)_k(d)_k}{(2-d)_k}\dfrac{(2-2d)_k}{(4-2d)_k}\dfrac{1}{k!}
\sum_{j=0}^\infty\dfrac{(d+k)_j}{(4-2d+k)_j} \\
&\qquad\qquad\qquad\qquad\qquad\qquad\qquad\qquad\qquad\qquad\qquad\quad
\hyp32{d,4-3d,3-2d+j+k}{5-4d,4-2d+j+k}{1}(y_{1})^{4-5d}
\end{align*}
The integral over $y_1$ gives
\begin{multline*}
c_5(3)=\dfrac{\Gamma(1-d)}{(1-d)(5-5d)}\dfrac{\Gamma(2-2d)}{\Gamma(4-2d)}
\dfrac{\Gamma(4-3d)}{\Gamma(5-4d)}
\sum_{k=0}^\infty\dfrac{(1-d)_k(d)_k}{(2-d)_k}\dfrac{(2-2d)_k}{(4-2d)_k}\dfrac{1}{k!}
\sum_{j=0}^\infty\dfrac{(d+k)_j}{(4-2d+k)_j} \\
\hyp32{d,4-3d,3-2d+j+k}{5-4d,4-2d+j+k}{1}
\end{multline*}
To work out the summations we do the following steps
\begin{description}
\item[-] Apply the Thomae transformation \eqref{55.1} and write the hypergeometric function as a summation over $m$.
\item[-] Work out the summation over $j$.
\item[-] The hypergeometric function can written as a fraction of Gamma functions. The result is a hypergeometric function.
\item[-] Apply the Thomae transformation \eqref{55.1} and write the hypergeometric function as a summation over $m$.
\item[-] Work out the summation over $k$.
\item[-] The result is the sum of two summations which can be written as the sum of two hypergeometric functions.
\end{description}
The result is
\begin{multline*}
c_5(3)=\dfrac{\Gamma(1-d)^{3}}{(1-d)\Gamma(2-2d)}\dfrac{\Gamma(4-5d)}{\Gamma(4-4d)}
\dfrac{\Gamma(3-3d)}{\Gamma(6-5d)}\hyp32{d,1-d,3-3d}{2-d,4-4d}{1}- \\
-\dfrac{\Gamma(1-d)^2}{(1-d)}
\dfrac{\Gamma(4-5d)\Gamma(2-2d)}{\Gamma(6-5d)\Gamma(6-5d)}
\hyp32{d,1-d,2-2d}{2-d,4-4d}{1}
\end{multline*}

\subsection{Derivation of $c_5(4)$}
For $c_5(4)$ we have
\begin{multline*}
c_5(4)=\int_0^1\int_0^{y_1}(y_1-y_2)^{-d}\int_0^{y_2}\int_0^{y_3}
(y_2-y_4)^{-d}(y_3-y_4)^{-d} \\
\int_0^{y_4}(y_3-y_5)^{-d}(y_1-y_5)^{-d}dy_5\, dy_4\, dy_3\,  dy_2\, dy_1
\end{multline*}
Using \eqref{1.1} and using the transformation $y_4=(y_3-y_5)z+y_5$ gives
\begin{multline*}
c_5(4)=\int_0^1\int_0^{y_1}(y_1-y_2)^{-d}\int_0^{y_2}\int_0^{y_3}(y_2-y_5)^{-d}
(y_3-y_5)^{1-2d}(y_1-y_5)^{-d} \\
\int_0^1(1-z)^{-d}\left( 1-\dfrac{y_3-y_5}{y_2-y_5}z\right)^{-d}dzdy_5dy_3dy_2dy_1
\end{multline*}
The integral over $z$ gives a hypergeometric function. Writing this as a summation gives
\begin{equation}
\int_0^1(1-z)^{-d}\left(1-\dfrac{y_3-y_5}{y_2-y_5}z\right)^{-d}dz=
\dfrac{1}{(1-d)}\sum_{k=0}^\infty\dfrac{(d)_k}{(2-d)_k}
\left(\dfrac{y_3-y_5}{y_2-y_5}\right)^k
\label{53.4.1}
\end{equation}
For the integral over $y_5$ we write the factor $(y_1-y_5)^{-d}$ as a summation and get after integration
\begin{align*}
&\dfrac{1}{(1-d)}\sum_{k=0}^\infty\dfrac{(d)_k}{(2-d)_k}
\sum_{j=0}^\infty(y_1)^{-d-j}\dfrac{(d)_j}{j!}
\int_0^{y_3}(y_2-y_5)^{-d-k}(y_3-y_5)^{1-2d+k}(y_5)^j dy_5 \\
&=\dfrac{1}{(1-d)}\sum_{k=0}^\infty\dfrac{(d)_k\Gamma(2-2d+k)}{(2-d)_k}
\sum_{j=0}^\infty(y_1)^{-d-j}\dfrac{(d)_j}{\Gamma(3-2d+j+k)}(y_2)^{-d-k}(y_3)^{2-2d+j+k} \\
&\qquad\qquad\qquad\qquad\qquad\qquad\qquad\qquad\qquad\qquad\qquad\qquad\qquad
\hyp21{1+j,d+k}{3-2d+j+k}{\dfrac{y_{3}}{y_{2}}}
\end{align*}
For the integral over $y_3$ we get
\begin{align*}
&\dfrac{1}{(1-d)}\sum_{k=0}^\infty\dfrac{(d)_k\Gamma(2-2d+k)}{(2-d)_k}
\sum_{j=0}^\infty(y_1)^{-d-j}\dfrac{(d)_j}{\Gamma(3-2d+j+k)}(y_2)^{-d-k} \\
&\qquad\qquad\qquad\qquad\qquad\qquad\qquad\qquad\qquad
\int_0^{y_2}(y_3)^{2-2d+j+k}\hyp21{1+j,d+k}{3-2d+j+k}{\dfrac{y_{3}}{y_{2}}}dy_3 \\
&=\dfrac{1}{(1-d)}\dfrac{\Gamma(3-3d)\Gamma(2-2d)}{\Gamma(3-2d)\Gamma(4-3d)}
\sum_{k=0}^\infty\dfrac{(d)_k(2-2d)_k}{(2-d)_k(3-2d)_k}
\sum_{j=0}^\infty(y_1)^{-d-j}\dfrac{(d)_j}{(4-3d)_j}(y_2)^{3-3d+j}
\end{align*}
For the integral over $y_2$ we get
\begin{align*}
&\dfrac{1}{(1-d)}\dfrac{\Gamma(3-3d)\Gamma(2-2d)}{\Gamma(3-2d)\Gamma(4-3d)}
\sum_{k=0}^\infty\dfrac{(d)_k(2-2d)_k}{(2-d)_k(3-2d)_k}
\sum_{j=0}^\infty(y_1)^{-d-j}\dfrac{(d)_j}{(4-3d)_j} \\
&\qquad\qquad\qquad\qquad\qquad\qquad\qquad\qquad\qquad\qquad\qquad\qquad\qquad
\int_{0}^{y_1}(y_1-y_2)^{-d}(y_2)^{3-3d+j}dy_2 \\
&=\dfrac{\Gamma(1-d)}{2(1-d)^2}\dfrac{\Gamma(3-3d)\Gamma(4-5d)}{\Gamma(4-4d)\Gamma(5-5d)}
\sum_{k=0}^{\infty}\dfrac{(d)_k(2-2d)_k}{(2-d)_k(3-2d)_k}(y_1)^{4-5d}
\end{align*}
Integration over $y_1$ gives at last
\[
c_{5}(4)=\dfrac{\Gamma(1-d)}{2(1-d)^2}\dfrac{\Gamma(3-3d)\Gamma(4-5d)}
{\Gamma(4-4d)\Gamma(6-5d)}\hyp32{1,d,2-2d}{2-d,3-2d}{1}
\]

\subsection{Derivation of $c_5(5)$}
For $c_5(5)$ we have
\begin{multline*}
c_5(5)=\int_0^1\int_0^{y_1}(y_1-y_2)^{-d}\int_0^{y_2}\int_0^{y_3}
(y_1-y_4)^{-d}(y_3-y_4)^{-d} \\
\int_0^{y_4}(y_3-y_5)^{-d}(y_1-y_5)^{-d}dy_5\, dy_4\, dy_3\,  dy_2\, dy_1
\end{multline*}
Using \eqref{1.1} and using the transformation $y_4=(y_3-y_5)z+y_5$ gives
\begin{multline*}
c_5(5)=\int_0^1\int_0^{y_1}(y_1-y_2)^{-d}\int_0^{y_2}\int_0^{y_3}(y_2-y_5)^{-d}
(y_3-y_5)^{1-2d}(y_1-y_5)^{-d} \\
\int_0^1(1-z)^{-d}\left( 1-\dfrac{y_3-y_5}{y_1-y_5}z\right)^{-d}dz\, dy_5\, dy_3\, dy_2\, dy_1
\end{multline*}
The integral over $z$ gives a hypergeometric function. Writing this as a summation gives
\begin{equation}
\int_0^1(1-z)^{-d}\left(1-\dfrac{y_3-y_5}{y_1-y_5}z\right)^{-d}dz=
\dfrac{1}{(1-d)}\sum_{k=0}^\infty\dfrac{(d)_k}{(2-d)_k}
\left(\dfrac{y_3-y_5}{y_1-y_5}\right)^k
\label{53.5.1}
\end{equation}
For the integral over $y_5$ we write the factor $(y_2-y_5)^{-d}$ as a summation and get
\begin{align*}
&\dfrac{1}{(1-d)}\sum_{k=0}^\infty\dfrac{(d)_k}{(2-d)_k}
\sum_{j=0}^\infty\dfrac{(d)_j}{j!}(y_2)^{-d-j}
\int_0^{y_3}(y_3-y_5)^{1-2d+k}(y_1-y_5)^{-d-k}(y_5)^j dy_5 \\
&=\dfrac{1}{(1-d)}\sum_{k=0}^\infty\dfrac{(d)_k\Gamma(2-2d+k)}{(2-d)_k}
\sum_{j=0}^\infty\dfrac{(d)_j}{\Gamma(3-2d+j+k)}(y_1)^{-d-k}(y_2)^{-d-j}(y_3)^{2-2d+j+k}\\
&\qquad\qquad\qquad\qquad\qquad\qquad\qquad\qquad\qquad\qquad\qquad\qquad\qquad\qquad
\hyp21{1+j,d+k}{3-2d+j+k}{\dfrac{y_{3}}{y_{1}}}
\end{align*}
For the integral over $y_3$ we get
\begin{align*}
&\dfrac{1}{(1-d)}\sum_{k=0}^\infty\dfrac{(d)_k\Gamma(2-2d+k)}{(2-d)_k}
\sum_{j=0}^\infty\dfrac{(d)_j}{\Gamma(3-2d+j+k)}(y_1)^{-d-k}(y_2)^{-d-j}\\
&\qquad\qquad\qquad\qquad\qquad\qquad\qquad\qquad\qquad\qquad
\int_0^{y_2}(y_3)^{2-2d+j+k}\hyp21{1+j,d+k}{3-2d+j+k}{\dfrac{y_{3}}{y_{1}}}dy_3 
\end{align*}
\begin{align*}
&=\dfrac{1}{(1-d)}\sum_{k=0}^\infty\dfrac{(d)_k\Gamma(2-2d+k)}{(2-d)_k}
\sum_{j=0}^\infty\dfrac{(d)_j}{\Gamma(4-2d+j+k)}(y_1)^{-d-k}(y_2)^{3-3d+k} \\
&\qquad\qquad\qquad\qquad\qquad\qquad\qquad\qquad\qquad\qquad\qquad\qquad\qquad\qquad\quad
\hyp21{1+j,d+k}{4-2d+j+k}{\dfrac{y_2}{y_1}}
\end{align*}
For the integral over $y_2$ we get
\begin{align*}
&\dfrac{1}{(1-d)}\sum_{k=0}^\infty\dfrac{(d)_k\Gamma(2-2d+k)}{(2-d)_k}
\sum_{j=0}^\infty\dfrac{(d)_j}{\Gamma(4-2d+j+k)}(y_1)^{-d-k}(y_2)^{3-3d+k} \\
&\qquad\qquad\qquad\qquad\qquad\qquad\qquad\qquad\qquad\qquad
\int_0^{y_1}(y_1-y_2)^{-d}\hyp21{1+j,d+k}{4-2d+j+k}{\dfrac{y_2}{y_1}}dy_2 
\end{align*}
\begin{align*}
&=\dfrac{\Gamma(1-d)}{(1-d)}\sum_{k=0}^\infty\dfrac{(d)_k\Gamma(2-2d+k)}
{(2-d)_k\Gamma(4-2d+k)}\dfrac{\Gamma(4-3d+k)}{\Gamma(5-4d+k)}
\sum_{j=0}^\infty\dfrac{(d)_j}{(4-2d+k)_j}(y_1)^{4-5d} \\
&\qquad\qquad\qquad\qquad\qquad\qquad\qquad\qquad\qquad\qquad\qquad\qquad
\hyp32{1+j,4-3d+k,d+k}{5-4d+k,4-2d+j+k}{1}
\end{align*}
The integral over $y_1$ gives
\begin{align*}
c_5(5)&=\dfrac{\Gamma(1-d)}{(1-d)(5-5d)}
\dfrac{\Gamma(2-2d)\Gamma(4-3d)}{\Gamma(4-2d)\Gamma(5-4d)}
\sum_{k=0}^\infty\dfrac{(d)_k(2-2d)_k(4-3d)_k}{(2-d)_k(4-2d)_k}(5-4d)_k \\
&\qquad\qquad\qquad\qquad\qquad\qquad\qquad\quad
\sum_{j=0}^\infty\dfrac{(d)_j}{(4-2d+k)_j}\hyp32{1+j,4-3d+k,d+k}{5-4d+k,4-2d+j+k}{1}
\end{align*}
To work out the summations we do the following steps
\begin{description}
	\item[-] Write the hypergeometric function as a summation over $m$.
	\item[-] Work out the summation over $j$.
	\item[-] The hypergeometric function can written as a fraction of Gamma functions. The result is a summation over $m$.
	\item[-] Work out the summation over $m$.
	\item[-] There remains a summation over $k$ which can be written as a hypergeometric function.
\end{description}
The result is
\[
c_5(5)=\dfrac{\Gamma(1-d)}{2(1-d)^2}
\dfrac{\Gamma(3-3d)\Gamma(4-5d)}{\Gamma(4-4d)\Gamma(6-5d)}
\hyp43{1,d,2-2d,3-3d}{2-d,3-2d,4-4d}{1}
\]

\subsection{Derivation of $c_5(6)$}
For $c_5(6)$ we have
\begin{multline*}
c_5(6)=\int_0^1\int_0^{y_1}\int_0^{y_2}(y_1-y_3)^{-d}\int_0^{y_3}
(y_3-y_4)^{-d}(y_2-y_4)^{-d} \\
\int_0^{y_4}(y_2-y_5)^{-d}(y_1-y_5)^{-d}dy_5\, dy_4\, dy_3\,  dy_2\, dy_1	
\end{multline*}
Using \eqref{1.1} and the transformation $y_4=(y_3-y_5)z+y_5$ gives
\begin{multline*}
c_5(6)=\int_0^1\int_0^{y_1}\int_0^{y_2}(y_1-y_3)^{-d}
\int_0^{y_3}(y_2-y_5)^{-2d}(y_1-y_5)^{-d}(y_3-y_5)^{1-d} \\
\int_0^1(1-z)^{-d}\left(1-\dfrac{y_3-y_5}{y_2-y_5}z\right)^{-d}dz\, dy_5\, dy_3\,  dy_2\, dy_1	
\end{multline*}
For the integral over $z$ we use \eqref{53.4.1}. The result gives a hypergeometric function. We write the factor $(y_1-y_5)^{-d}$ and the hypergeometric function as summations . After integration we get
\begin{multline*}
\dfrac{1}{(1-d)}\sum_{k=0}^\infty\dfrac{(d)_k}{(2-d)_k}
\sum_{j=0}^\infty\dfrac{(d)_j}{j!}(y_1)^{-d-j}
\int_0^{y_3}(y_2-y_5)^{-2d-k}(y_3-y_5)^{1-d+k}(y_5)^j dy_5= \\
=\dfrac{\Gamma(1-d)}{\Gamma(3-d)}\sum_{k=0}^\infty\dfrac{(d)_k}{(3-d)_k}
\sum_{j=0}^\infty\dfrac{(d)_j}{(3-d+k)_j}(y_1)^{-d-j}(y_2)^{-2d-k}(y_3)^{2-d+j+k}
\hyp21{1+j,2d+k}{3-d+j+k}{\dfrac{y_3}{y_2}}
\end{multline*}
Writing the hypergeometric function as a summation, using the transformations from Remark 5 and integration over $y_1$ gives
\begin{multline*}
c_5(6)=\dfrac{1}{(5-5d)}\dfrac{\Gamma(1-d)}{\Gamma(3-d)}
\sum_{k=0}^\infty\dfrac{(d)_k}{(3-d)_k}\sum_{j=0}^\infty\dfrac{(d)_j}{(3-d+k)_j}
\sum_{i=0}^\infty\dfrac{(1+j)_i(2d+k)_i}{(3-d+j+k)_i}\dfrac{1}{i!} \\
\int_0^1\int_0^1(z_2)^{3-3d+j}(z_3)^{2-d+i+j+k}(1-z_2\, z_3)^{-d}dz_3\, dz_2
\end{multline*}
For the double integral we get
\begin{multline*}
\int_0^1\int_0^1(z_2)^{3-3d+j}(z_3)^{2-d+i+j+k}(1-z_2\, z_3)^{-d}dz_3\, dz_2= \\
=\dfrac{\Gamma(1-d)}{(2d-1+i+k)}\left(\dfrac{\Gamma(4-3d+j)}{\Gamma(5-4d+j)}-
\dfrac{\Gamma(3-d+i+j+k)}{\Gamma(4-2d+i+j+k)}\right)
\end{multline*}
Using the summations consecutively over $i$, $j$ and $k$ gives
\begin{multline*}
c_5(6)=\dfrac{\Gamma(1-d)}{(1-d)}\dfrac{\Gamma(2d-1)}{\Gamma(2d)}
\dfrac{\Gamma(3-3d)\Gamma(4-5d)}{\Gamma(6-5d)\Gamma(4-4d)}\hyp32{1,d,2d-1}{2-d,2d}{1}- \\
-\dfrac{\Gamma(1-d)^2}{(3-2d)}\dfrac{\Gamma(2d-1)}{\Gamma(2d)}
\dfrac{\Gamma(4-5d)}{\Gamma(6-5d)}\hyp32{2d-1,1,d}{2d,3-2d}{1}
\end{multline*}
Application of the Thomae transformations \eqref{55.3} and \eqref{55.5} gives
\begin{align*}
c_5(6)&=\dfrac{\Gamma(1-d)}{(1-d)}\dfrac{\Gamma(2d-1)}{\Gamma(2d)}
\dfrac{\Gamma(3-3d)\Gamma(4-5d)}{\Gamma(6-5d)\Gamma(4-4d)}\hyp32{1,d,2d-1}{2-d,2d}{1}- \\
&-\dfrac{\Gamma(1-d)^3\Gamma(4-5d)}{\Gamma(4-4d)\Gamma(6-5d)}
\dfrac{\Gamma(2-2d)\Gamma(2d-1)}{\Gamma(d)}+ \\
&+\dfrac{\Gamma(1-d)^2}{3(1-d)^2}\dfrac{\Gamma(4-5d)}{\Gamma(6-5d)\Gamma(2-2d)}
\hyp32{1,d,3-3d}{2-d,4-3d}{1}
\end{align*}

\subsection{Derivation of $c_5(7)$}
For $c_5(7)$ we have
\begin{multline*}
c_5(7)=\int_0^1\int_0^{y_1}\int_0^{y_2}(y_1-y_3)^{-d}\int_0^{y_3}
(y_2-y_4)^{-d}(y_1-y_4)^{-d} \\
\int_0^{y_4}(y_2-y_5)^{-d}(y_3-y_5)^{-d}dy_5\, dy_4\, dy_3\,  dy_2\, dy_1
\end{multline*}
Using \eqref{1.1} and using the transformation $y_4=(y_3-y_5)z+y_5$ gives
\begin{multline*}
c_5(7)=\int_0^1\int_0^{y_1}\int_0^{y_2}(y_1-y_3)^{-d}
\int_0^{y_3}(y_2-y_5)^{-2d}(y_3-y_5)^{1-d}(y_1-y_5)^{-d} \\
\int_0^1\left(1-\dfrac{y_3-y_5}{y_2-y_5}z\right)^{-d}
\left(1-\dfrac{y_3-y_5}{y_1-y_5}z\right)^{-d}dz\, dy_5\, dy_3\, dy_2\, dy_1
\end{multline*}
Making use of the known integral
\begin{multline}
\int_0^1(1-p\, z)^{-d}(1-q\, z)^{-d}dz= \\
=\dfrac{1}{(1-d)}\dfrac{1}{p}\hyp21{1,d}{2-d}{\dfrac{q}{p}}
-\dfrac{1}{(1-d)}\dfrac{1}{p}(1-p)^{1-d}(1-q)^{-d}
\hyp21{1,d}{2-d}{\dfrac{q(1-p)}{p(1-q)}}
\label{53.7.1}
\end{multline}
we get
\begin{multline*}
\int_0^1\left(1-\dfrac{y_3-y_5}{y_2-y_5}z\right)^{-d}
\left(1-\dfrac{y_3-y_5}{y_1-y_5}z\right)^{-d}dz 
=\dfrac{1}{(1-d)}\dfrac{(y_2-y_5)}{(y_3-y_5)}
\hyp21{1,d}{2-d}{\dfrac{y_2-y_5}{y_1-y_5}}- \\
-\dfrac{1}{\left( 1-d\right) }%
\dfrac{(y_2-y_5)}{(y_3-y_5)}\left(\dfrac{y_2-y_3}{y_2-y_5}\right)^{1-d}
\left(\dfrac{y_1-y_3}{y_1-y_5}\right)^{-d}\hyp21{1,d}{2-d}{\dfrac{y_2-y_3}{y_1-y_3}}
\end{multline*}
After substitution in $c_5(7)$ with
\[
\int_0^{y_3}(y_3-y_5)^{-d}(y_2-y_5)^{-d}dy_5=\dfrac{1}{(1-d)}
\hyp21{1,d}{2-d}{\dfrac{y_3}{y_2}}
\]
we get \quad $c_5(7)=I_1-I_2$ with
\begin{multline*}
I_1=\dfrac{1}{(1-d)}\int_0^1\int_0^{y_1}\int_0^{y_2}
(y_1-y_3)^{-d}\int_0^{y_3}(y_1-y_5)^{-d}(y_2-y_5)^{1-2d}(y_3-y_5)^{-d} \\
\hyp21{1,d}{2-d}{\dfrac{y_2-y_5}{y_1-y_5}}dy_{5}dy_{3}dy_{2}dy_{1}
\end{multline*}
\begin{multline*}
I_2=\dfrac{1}{(1-d)^2}\int_0^1\int_0^{y_1}\int_0^{y_2}
(y_2)^{-d}(y_3)^{1-d}(y_1-y_3)^{-2d}(y_2-y_3)^{1-d}
\hyp21{1,d}{2-d}{\dfrac{y_3}{y_2}} \\
\hyp21{1,d}{2-d}{\dfrac{y_2-y_3}{y_1-y_3}}dy_{3}dy_{2}dy_{1}
\end{multline*}

\subsubsection{Derivation of $I_1$.}
Writing the factor $(y_1-y_5)^{-d}$ and the hypergeometric function as summations we get
\begin{multline*}
I_1=\dfrac{1}{(1-d)}\sum_{k=0}^\infty\sum_{j=0}^\infty
\dfrac{(d)_k}{(2-d)_k}\dfrac{(d+k)_j}{j!} \\
\int_0^1\int_0^{y_1}\int_0^{y_2}(y_1-y_3)^{-d}\int_0^{y_3}
(y_5)^j(y_2-y_5)^{1-2d+k}(y_3-y_5)^{-d}dy_5\, dy_3\, dy_2\, dy_1
\end{multline*}
Integration over $y_5$ gives a hypergeometric function. Writing this function as a summation over $i$ gives
\begin{multline*}
I_1=\dfrac{\Gamma(1-d)}{(1-d)}\sum_{k=0}^\infty\sum_{j=0}^\infty\sum_{i=0}^\infty
\dfrac{(d)_k(d+k)_j}{(2-d)_k\Gamma(2-d+j)}\dfrac{(1+j)_i(2d-1-k)_i}{(2-d+j)_i\, i!} \\
\int_0^1\int_0^{y_1}\int_0^{y_2}(y_1)^{-d-j-k}(y_2)^{1-2d-i+k}(y_3)^{1-d+i+j}(y_1-y_3)^{-d}dy_3\, dy_2\, dy_1
\end{multline*}
The triple integral can simply be done and we get
\begin{align*}
I_1&=\dfrac{\Gamma(1-d)^2}{(1-d)(5-5d)}\sum_{k=0}^\infty\sum_{j=0}^\infty\sum_{i=0}^\infty
\dfrac{\Gamma(4-3d+j+k)(d)_k(1+j)_i(2d-1-k)_i(d+k)_j}
{(2d-2+i-k)i!\Gamma(2-d_i+j)(2-d)_k\Gamma(5-4d+j+k)}- \\
&-\dfrac{\Gamma(1-d)^2}{(1-d)(5-5d)}\sum_{k=0}^\infty\sum_{j=0}^\infty\sum_{i=0}^\infty
\dfrac{(d)_k(1+j)_i\Gamma(2d-2-k)(2d-2-k)_i(d+k)_j}
{\Gamma(2d-2-k)i!(2-d)_k\Gamma(3-2d+i+j)}
\end{align*}
After doing consecutively the summations over $i$,\, $j$ and $k$ gives at last
\begin{multline*}
I_{1}=\dfrac{\Gamma(1-d)^2}{2(1-d)^2}\dfrac{\Gamma( 4-5d)}{\Gamma(6-5d)\Gamma(2-2d)}
\hyp32{1,d,2-2d}{2-d,3-2d}{1}- \\
-\dfrac{\Gamma(1-d)}{2(1-d)^2}
\dfrac{\Gamma(3-3d)}{\Gamma(4-4d)}\dfrac{\Gamma(4-5d)}{\Gamma(6-5d)}
\hyp43{1,d,2-2d,3-3d}{2-d,3-2d,4-4d}{1}
\end{multline*}

\subsubsection{Derivation of $I_2$.}
Using the transformations of Remark 5, integration over $y_1$ and writing the hypergeometric functions as summations gives
\begin{multline*}
I_2=\dfrac{1}{(1-d)^2(5-5d)}
\sum_{k=0}^\infty\dfrac{(d)_k}{(2-d)_k}\sum_{j=0}^\infty\dfrac{(d)_j}{(	2-d)_j} \\
\int_0^1\int_0^1(z_2)^{3-3d+j}(z_3)^{1-d+k}(1-z_2\, z_3)^{-2d-j}(1-z_3)^{1-d+j}dz_3\, dz_2
\end{multline*}
The integral over $z_2$ is known just like the integral over $z_3$. The result is
\begin{multline*}
I_2=\dfrac{\Gamma(3-3d)}{(1-d)^2(5-5d)}
\sum_{k=0}^\infty\dfrac{(d)_k}{(2-d)_k}\sum_{j=0}^\infty\dfrac{(d)_j}{(2-d)_j}
\dfrac{\Gamma(4-3d+j)}{\Gamma(5-3d+j)}\dfrac{\Gamma(2-d+k)}{\Gamma(5-4d+k)} \\
\hyp32{1,5-5d,2-d+k}{5-3d+j,5-4d+k}{1}
\end{multline*}
Application of the Thomae transformation \eqref{55.1} gives
\begin{multline*}
I_2=\dfrac{\Gamma(3-3d)}{(1-d)^2(5-5d)}
\sum_{k=0}^\infty\dfrac{(d)_k}{(2-d)_k}\sum_{j=0}^\infty\dfrac{(d)_j}{(2-d)_j}
\dfrac{\Gamma(4-3d+j)}{\Gamma(5-3d+j)}\dfrac{\Gamma(2-d+k)}{\Gamma(5-4d+k)} \\
\dfrac{\Gamma(5-3d+j)\Gamma(2-d+j)}{\Gamma(4-3d+j)\Gamma(3-d+j)}
\hyp32{1,3-3d,d+k}{3-d+j,5-4d+k}{1}
\end{multline*}
Writing the hypergeometric function as a summation and after a lot of summations we get at last
\[
I_2=\dfrac{\Gamma(1-d)}{2(1-d)^2}\dfrac{\Gamma(3-3d)}{\Gamma(6-5d)}
\dfrac{\Gamma(4-5d)}{\Gamma(4-4d)}\hyp43{1,d,2-2d,3-3d}{2-d,3-2d,4-4d}{1}
\]
Combining $I_1$ and $I_2$ gives for $c_5(7)$
\begin{multline*}
c_5(7)=\dfrac{\Gamma(1-d)^2}{2(1-d)^2}\dfrac{\Gamma(4-5d)}{\Gamma(6-5d)\Gamma(2-2d)}
\hyp32{1,d,2-2d}{2-d,3-2d}{1}- \\
\dfrac{\Gamma(1-d)}{(1-d)^2}\dfrac{\Gamma(3-3d)}{\Gamma(4-4d)}
\dfrac{\Gamma(4-5d)}{\Gamma(6-5d)}
\hyp43{1,d,2-2d,3-3d}{2-d,3-2d,4-4d}{1}
\end{multline*}

\subsection{Derivation of $c_5(8)$}
For $c_5(8)$ we have
\begin{multline*}
c_5(8)=\int_0^1\int_0^{y_1}\int_0^{y_2}(y_2-y_3)^{-d}\int_0^{y_3}
(y_3-y_4)^{-d}(y_1-y_4)^{-d} \\
\int_0^{y_4}(y_1-y_5)^{-d}(y_2-y_5)^{-d}dy_5\, dy_4\, dy_3\,  dy_2\, dy_1
\end{multline*}
Using \eqref{1.1} and the transformation $y_4=(y_3-y_5)z+y_5$ gives
\begin{multline*}
c_5(8)=\int_0^1\int_0^{y_1}\int_0^{y_2}(y_2-y_3)^{-d}
\int_0^{y_3}(y_1-y_5)^{-2d}(y_2-y_5)^{-d}(y_3-y_5)^{1-d} \\
\int_0^1(1-z)^{-d}\left(1-\dfrac{y_3-y_5}{y_1-y_5}z\right)^{-d}dz\, dy_5\, dy_3\,  dy_2\, dy_1	
\end{multline*}
Using \eqref{53.5.1} for the integral over $z$ and writing the factor $(y_1-y_5)^{-d}$ as a summation we get for the integral over $y_5$
\begin{align*}
&\dfrac{1}{(1-d)}\sum_{k=0}^\infty\dfrac{(d)_k}{(2-d)_k}
\sum_{j=0}^\infty\dfrac{(2d+k)_j}{j!}(y_1)^{-2d-k-j}
\int_0^{y_3}(y_2-y_5)^{-d}(y_3-y_5)^{1-d+k}(y_5)^jdy_5= \\
&=\dfrac{1}{(1-d)}\sum_{k=0}^\infty\dfrac{(d)_k}{(2-d)_k}
\sum_{j=0}^\infty\dfrac{(2d+k)_j\Gamma(2-d+k)}{\Gamma(3-d+j+k)}
(y_1)^{-2d-k-j}(y_2)^{-d}(y_3)^{2-d+j+k} \\
&\qquad\qquad\qquad\qquad\qquad\qquad\qquad\qquad\qquad\qquad\qquad\qquad\qquad\qquad\quad
\hyp21{d,1+j}{3-d+j+k}{\dfrac{y_3}{y_2}}
\end{align*}
Integration over $y_3$ gives
\begin{align*}
&\dfrac{1}{(1-d)}\sum_{k=0}^\infty\dfrac{(d)_k}{(2-d)_k}
\sum_{j=0}^\infty\dfrac{(2d+k)_j\Gamma(2-d+k)}{\Gamma(3-d+j+k)}
(y_1)^{-2d-k-j}(y_2)^{-d} \\
&\qquad\qquad\qquad\qquad\qquad\qquad\qquad\quad
\int_0^{y_2}(y_3)^{2-d+j+k}(y_2-y_3)^{-d}\hyp21{d,1+j}{3-d+j+k}{\dfrac{y_3}{y_2}}dy_3= 
\end{align*}
\begin{align*}
&=\dfrac{\Gamma(1-d)^2\Gamma(3-3d)}{\Gamma(3-2d)}
\sum_{k=0}^\infty\dfrac{(d)_k(3-3d)_k}{(3-2d)_k}
\sum_{j=0}^\infty\dfrac{(2d+k)_j}{\Gamma(4-3d+j+k)}(y_1)^{-2d-k-j}(y_2)^{3-3d+j+k}
\end{align*}
Integration over $y_2$ and $y_1$ gives at last
\[
c_5(8)=\dfrac{\Gamma(1-d)^2}{3(1-d)}\dfrac{\Gamma(4-5d)}{\Gamma(6-5d)\Gamma(3-2d)}
\hyp32{1,d,3-3d}{3-2d,4-3d}{1}
\]

\subsection{Derivation of $c_5(9)$}
For $c_5(9)$ we have
\begin{multline*}
c_5(9)=\int_0^1\int_0^{y_1}(y_1-y_2)^{-d}\int_0^{y_2}(y_1-y_3)^{-d}\int_0^{y_3}
(y_3-y_4)^{-d} \\
\int_0^{y_4}(y_2-y_5)^{-d}(y_4-y_5)^{-d}dy_5\, dy_4\, dy_3\,  dy_2\, dy_1
\end{multline*}
For the integral over $y_5$ we get
\[
\int_0^{y_4}(y_2-y_5)^{-d}(y_4-y_5)^{-d}dy_{5}=
\dfrac{1}{(1-d)}(y_2)^{-d}(y_4)^{1-d}\hyp21{1,d}{2-d}{\dfrac{y_4}{y_2}}
\]
For the integral over $y_4$ we get
\begin{multline*}
\dfrac{1}{(1-d)}(y_2)^{-d}\int_0^{y_3}(y_4)^{1-d}(y_3-y_4)^{-d}
\hyp21{1,d}{2-d}{\dfrac{y_4}{y_2}} dy_4= \\
=\dfrac{1}{(1-d)}(y_2)^{-d}\dfrac{\Gamma(1-d)\Gamma(2-d)}{\Gamma(3-2d)}(y_3)^{2-2d}
\hyp21{1,d}{2-d}{\dfrac{y_3}{y_2}}
\end{multline*}
For the integral over $y_3$ we write the factor $(y_1-y_3)^{-d}$ as a summation and get
\begin{multline*}
\dfrac{\Gamma(1-d)^2}{\Gamma(3-2d)}(y_2)^{-d}
\sum_{k=0}^\infty\dfrac{(d)_k}{k!}(y_1)^{-d-k}\int_0^{y_2}(y_3)^{2-2d+k}
\hyp21{1,d}{2-d}{\dfrac{y_3}{y_2}}dy_3 \\
=\dfrac{\Gamma(1-d)^2}{\Gamma(3-2d)}
\sum_{k=0}^\infty\dfrac{(d)_k}{k!}\dfrac{1}{(3-2d+k)}\hyp32{1,d,3-2d+k}{3-2d,4-2d+k}{1}
(y_1)^{-d-k}(y_2)^{3-3d+k}
\end{multline*}
For the integral over $y_2$ we get
\begin{align*}
&\dfrac{\Gamma(1-d)^2}{\Gamma(3-2d)}
\sum_{k=0}^\infty\dfrac{(d)_k}{k!}\dfrac{1}{(3-2d+k)}\hyp32{1,d,3-2d+k}{3-2d,4-2d+k}{1}
(y_1)^{-d-k} \\
&\qquad\qquad\qquad\qquad\qquad\qquad\qquad\qquad\qquad\qquad\qquad\qquad\qquad
\int_0^{y_1}(y_2)^{3-3d+k}(y_1-y_2)^{-d}dy_2  \\
&=\dfrac{\Gamma(1-d)^2}{\Gamma(3-2d)}
\sum_{k=0}^\infty\dfrac{(d)_k}{k!}\dfrac{1}{(3-2d+k)}\hyp32{1,d,3-2d+k}{3-2d,4-2d+k}{1} \dfrac{\Gamma(4-3d+k)}{\Gamma(5-4d+k)}(y_1)^{4-5d}
\end{align*}
After integration over $y_1$ we get
\[
c_5(9)=\dfrac{\Gamma(1-d)^2}{(5-5d)\Gamma(3-2d)}
\sum_{k=0}^\infty\dfrac{(d)_k}{k!}\dfrac{1}{(3-2d+k)}\hyp32{1,d,3-2d+k}{3-2d,4-2d+k}{1} \dfrac{\Gamma(4-3d+k)}{\Gamma(5-4d+k)}
\]

To work out the summation we do the following steps
\begin{description}
	\item[-] Write the hypergeometric function as a summation over $m$.
	\item[-] Work out the summation over $k$. The result is a summation over $m$ and a hypergeometric function.
	\item[-] Apply the Thomae transformation \eqref{55.1} to the hypergeometric function. 
	\item[-] Write the hypergeometric function as a summation over $k$.
	\item[-] Work out the summation over $m$.
	\item[-] There remains a summation over $k$ which can be written as a hypergeometric function.
\end{description}
The result is
\[
c_5(9)=\dfrac{\Gamma(1-d)^2}{(1-d)}
\dfrac{\Gamma(4-5d)\Gamma(2-2d)}{\Gamma(6-5d)\Gamma(4-4d)}
\hyp32{d,1-d,2-2d}{2-d,4-4d}{1}
\]
This result is the same as the result for $c_5(2)$. However, the starting integrals are different.

\subsection{Derivation of $c_5(10)$}
For $c_5(10)$ we have
\begin{multline*}
c_5(10)=\int_0^1\int_0^{y_1}\int_0^{y_2}(y_1-y_3)^{-d}(y_2-y_3)^{-d|}\int_0^{y_3}
(y_1-y_4)^{-d} \\
\int_0^{y_4}(y_2-y_5)^{-d}(y_4-y_5)^{-d}dy_5\, dy_4\, dy_3\,  dy_2\, dy_1 	
\end{multline*}
For the integral over $y_5$ we get
\begin{multline*}
c_5(10)=\dfrac{1}{(1-d)}\int_0^1\int_0^{y_1}\int_0^{y_2}(y_1-y_3)^{-d}(y_2-y_3)^{-d|} \\
\int_0^{y_3}(y_4)^{1-d}(y_1-y_4)^{-d}\hyp21{1,d}{2-d}{\dfrac{y_4}{y_2}}dy_4\, dy_3\,  dy_2\, dy_1 	
\end{multline*}
Writing the hypergeometric function as a summation over $k$ and integration over $y_4$ gives
\begin{multline*}
c_5(10)=\dfrac{1}{(1-d)}\sum_{k=0}^\infty\dfrac{(d)_k}{(2-d)_k}\int_0^1\int_0^{y_1}\int_0^{y_2}(y_1)^{-d}(y_2)^{-d-k}(y_3)^{2-d+k}(y_1-y_3)^{-d}(y_2-y_3)^{-d|} \\
\hyp21{d,2-d+k}{3-d+k}{\dfrac{y_3}{y_1}}dy_3\,  dy_2\, dy_1 	
\end{multline*}
Writing the hypergeometric function as a summation over $j$ and integration over $y_3$ gives
\begin{multline*}
c_5(10)=\dfrac{\Gamma(1-d)}{\Gamma(3-d)}\sum_{k=0}^\infty\sum_{j=0}^\infty
\dfrac{(d)_k}{(2-d)_k}\dfrac{(d)_j(2-d+k)_j}{(3-d)_k}\dfrac{(d)_j(2-d+k)_j}{(3-d+k)_j}
\dfrac{1}{j!} \\
\int_0^1\int_0^{y_1}\int_0^{y_2}(y_1)^{-d-j}(y_2)^{-d-k}(y_3)^{2-d+j+k}(y_1-y_3)^{-d}(y_2-y_3)^{-d}dy_3\,  dy_2\, dy_1 	
\end{multline*}
For the triple integral we get
\begin{multline*}
\int_0^1\int_0^{y_1}\int_0^{y_2}(y_1)^{-d-j}(y_2)^{-d-k}(y_3)^{2-d+j+k}(y_1-y_3)^{-d}(y_2-y_3)^{-d}dy_3\,  dy_2\, dy_1= \\
= \dfrac{\Gamma(1-d)}{(5-5d)}\dfrac{\Gamma(3-d+j+k)}{(4-3d+j)\Gamma(4-2d+j+k)}
\hyp32{d,4-3d+j,3-d+j+k}{5-3d+j,4-2d+j+k}{1}
\end{multline*}
Application of the Thomae transformation \eqref{55.5} gives for $c_5(10)$ after a lot of manipulations
\[
c_5(10)=\dfrac{\Gamma(1-d)\Gamma(2-2d)}{(1-d)(5-5d)}\sum_{k=0}^\infty\sum_{j=0}^\infty
\dfrac{(d)_k(d)_j(2-d+k)_j}{\Gamma(5-3d+j+k}\dfrac{1}{j!}
\hyp32{1,2-2d,d+k}{2-d,5-3d+j+k}{1}
\]

To work out the summations we do the following steps
\begin{description}
	\item[-] Write the hypergeometric function as a summation over $i$.
	\item[-] Work out the summation over $j$ and then the summation over $k$.. 
	\item[-] There remains a summation over $i$ which can be written as a hypergeometric function.
\end{description}
The result is
\[
c_5(10)=\dfrac{\Gamma(1-d)}{2(1-d)^2}
\dfrac{\Gamma(3-3d)\Gamma(4-5d)}{\Gamma(4-4d)\Gamma(6-5d)}
\hyp43{1,d,2-2d,3-3d}{2-d,3-2d,4-4d}{1}
\]
This result is the same as the result for $c_5(5)$. However, the starting integrals are different.

\subsection{Derivation of $c_5(11)$}
For $c_5(11)$ we have
\begin{multline*}
c_5(11)=\int_0^1\int_0^{y_1}\int_0^{y_2}(y_2-y_3)^{-d}\int_0^{y_3}
(y_1-y_4)^{-d}(y_2-y_4)^{-d} \\
\int_0^{y_4}(y_1-y_5)^{-d}(y_3-y_5)^{-d}dy_5\, dy_4\, dy_3\,  dy_2\, dy_1 
\end{multline*}
Using \eqref{1.1} and the transformation $y_4=(y_3-y_5)z+y_5$ gives
\begin{multline*}
c_5(11)=\int_0^1\int_0^{y_1}\int_0^{y_2}(y_2-y_3)^{-d}\int_0^{y_3}(y_1-y_5)^{-2d}
(y_3-y_5)^{1-d}(y_2-y_5)^{-d} \\
\int_0^1\left(1-\dfrac{y_3-y_5}{y_1-y_5}z\right)^{-d}
\left(1-\dfrac{y_3-y_5}{y_2-y_5}z\right)^{-d}dz\, dy_5\, dy_3\, dy_2\, dy_1
\end{multline*}
Using the integral \eqref{53.7.1} we get $c_5(11)=I_1-I_2$ with
\begin{multline*}
I_1=\dfrac{1}{(1-d)}\int_0^1\int_0^{y_1}\int_0^{y_2}\int_0^{y_3}
(y_2-y_3)^{-d}(y_1-y_5)^{-2d}(y_3-y_5)^{-d}(y_2-y_5)^{1-d} \\
\hyp21{1,d}{2-d}{\dfrac{y_{2}-y_{5}}{y_{1}-y_{5}}}dy_5\, dy_3\, dy_2\, dy_1
\end{multline*}
\begin{multline*}
I_{2}=\dfrac{1}{(1-d)}\int_0^1\int_0^{y_1}\int_0^{y_2}(y_2-y_3)^{1-2d}(y_1-y_3)^{-d}
\hyp21{1,d}{2-d}{\dfrac{y_{2}-y_{3}}{y_{1}-y_{3}}} \\
\int_0^{y_3}(y_1-y_5)^{-d}(y_3-y_5)^{-d}dy_5\, dy_3\, dy_2\, dy_1
\end{multline*}

\subsubsection{Derivation of $I_1$.}
Using the substitutions consecutively over $y_5=y_3\, z_3$, $y_3=y_2\, z_3$, and $y_2=y_1\, z_2$, integrating over $y_1$ and writing the hypergeometric functions as summations gives
\begin{multline*}
I_1=\dfrac{1}{(1-d)(5-5d)}\sum_{k=0}^\infty\dfrac{(d)_{k}}{(2-d)_{k}}
\int_0^1\int_0^1(1-z_3)^{-d}(z_3)^{1-d}(1-z_5)^{-d}(1-z_3\, z_5)^{1-d+k} \\
\int_0^1(z_2)^{3-3d+k}(1-z_2\, z_3\, z_5)^{-2d-k}dz_2\, dz_3\, dz_5
\end{multline*}

To work out the integrals and the summations we do the following steps
\begin{description}
	\item[-] Integration over $z_2$ gives a hypergeometric function.
	\item[-] Write the hypergeometric function as a summation over $j$.
	\item[-] Integration over $z_3$ gives a hypergeometric function.
	\item[-] Integrate over $z_5$.
	\item[-] Work out the summation over $j$.
	\item[-] There remains a summation over $k$ which can be written as a hypergeometric function.
\end{description}
The result is
\[
I_1=\dfrac{\Gamma(1-d)^{2}}{3(1-d)^2}\dfrac{\Gamma(4-5d)}{\Gamma(6-5d)\Gamma(2-2d)}
\hyp32{1,d,3-3d}{2-d,4-3d}{1}
\]

\subsubsection{Derivation of $I_2$.}
The integral over $y_5$ is known and we get
\begin{multline*}
I_2=\dfrac{1}{(1-d)^2}\int_0^1\int_0^{y_1}\int_0^{y_2}
(y_1)^{-d}(y_3)^{1-d}(y_2-y_)^{1-2d}(y_1-y_3)^{-d} \\
\hyp21{1,d}{2-d}{\dfrac{y_2-y_3}{y_1-y_3}}\hyp21{1,d}{2-d}{\dfrac{y_3}{y_1}}
dy_3\, dy_2\, dy_1
\end{multline*}
To work out the integrals we do the following steps
\begin{description}
	\item[-] Write the hypergeometric functions as a summation over $k$ and $j$
	\item[-] Do the integrations
	\item[-] Work out the summation over $j$
	\item[-] There remains a summation over $k$ which can be written as a hypergeometric function
\end{description}
\[
I_2=\dfrac{\Gamma(1-d)\Gamma(4-5d)\Gamma(3-3d)}{2(1-d)^2\Gamma(6-5d)\Gamma(4-4d)}
\hyp32{1,d,2-2d}{2-d,3-2d}{1}
\]
Combining $I_1$ and $I_2$ gives for $c_5(11)$
\begin{multline*}
c_5(11)=\dfrac{\Gamma(1-d)^{2}}{3(1-d)^2}\dfrac{\Gamma(4-5d)}{\Gamma(6-5d)\Gamma(2-2d)}
\hyp32{1,d,3-3d}{2-d,4-3d}{1}- \\
-\dfrac{\Gamma(1-d)\Gamma(4-5d)\Gamma(3-3d)}{2(1-d)^2\Gamma(6-5d)\Gamma(4-4d)}
\hyp32{1,d,2-2d}{2-d,3-2d}{1}
\end{multline*}

\subsection{Derivation of $c_5(12)$}
For $c_5(12)$ we have
\begin{multline*}
c_5(12)=\int_0^1\int_0^{y_1}\int_0^{y_2}(y_1-y_3)^{-d}(y_2-y_3)^{-d}\int_0^{y_3}
(y_2-y_4)^{-d}\\
\int_0^{y_4}(y_1-y_5)^{-d}(y_4-y_5)^{-d}dy_5\, dy_4\, dy_3\,  dy_2\, dy_1 	
\end{multline*}
Using \eqref{1.1} and the transformation $y_4=(y_3-y_5)z+y_5$ gives
\begin{multline*}
c_5(12)=\int_0^1\int_0^{y_1}\int_0^{y_2}(y_1-y_3)^{-d}(y_2-y_3)^{-d}
\int_0^{y_3}(y_1-y_5)^{-d}(y_3-y_5)^{1-d}(y_2-y_5)^{-d} \\
\int_0^1 z^{-d}\left( 1-\dfrac{y_3-y_5}{y_2-y_5}z\right)^{-d}dz\, dy_5\, dy_3\, dy_2\, dy_1
\end{multline*}
The integral over $z$ gives a hypergeometric function.
\[
\int_0^1 z^{-d}\left(1-\dfrac{y_3-y_5}{y_2-y_5}z\right)^{-d}dz=
\dfrac{1}{(1-d)}(y_3-y_5)^{1-d}(y_2-y_3)^{-d}
\hyp21{1,d}{2-d}{-\dfrac{y_{3}-y_{5}}{y_{2}-y_{3}}}
\]
Using the transformations from Remark 5, writing the factor $(1-z_2\, z_3\, z_5)^{-d}$ as a summation and integration over $y_1$ gives
\begin{multline*}
c_5(12)=\dfrac{1}{(1-d)(5-5d)}\sum_{k=0}^\infty\dfrac{(d)_k}{k!}
\int_0^1\int_0^1(z_2)^{3-3d+k}(1-z_2\, z_3)^{-d}(z_3)^{2-d+k}(1-z_3)^{-2d} \\
\int_0^1(z_5)^{1-d}(1-z_5)^{k}
\hyp21{1,d}{2-d}{-\dfrac{z_{3}}{1-z_{3}}z_{5}}dz_5\, dz_3\, dz_2
\end{multline*}
The integrals over $z_2$ and $z_5$ give
\begin{align*}
&\int_0^1(z_2)^{3-3d+k}(1-z_2\, z_3)^{-d}dz_2=
\dfrac{1}{(4-3d+k)}
\hyp21{d,4-3d+k}{5-3d+k}{z_3} \\
&\int_0^1(z_5)^{1-d}(1-z_5)^{k}\hyp21{1,d}{2-d}{-\dfrac{z_{3}}{1-z_{3}}z_{5}}dz_5=
\dfrac{\Gamma(2-d)\Gamma(k+1)}{\Gamma(3-d+k)}
(1-z_3)^d \hyp21{2-d+k,d}{3-d+k}{z_3}
\end{align*}
After substitution there remains
\begin{multline*}
c_5(12) =\dfrac{1}{(5-5d)}\dfrac{\Gamma(1-d)}{\Gamma(3-d)}
\sum_{k=0}^\infty\dfrac{(d)_k}{(3-d)_k}\dfrac{1}{(4-3d+k)} \\
\int_0^1(z_3)^{2-d+k}(1-z_3)^{-d}
\hyp21{2-d+k,d}{3-d+k}{z_3}\hyp21{d,4-3d+k}{5-3d+k}{z_3}dz_3
\end{multline*}
The integral is known \cite[2.21.9.(12)]{2}. Application gives
\begin{align*}
c_5(12)&=\dfrac{2}{(5-5d)}\dfrac{\Gamma(2-2d)\Gamma(1-d)^3}{\Gamma(5-4d)}
\dfrac{\Gamma(2d-2)}{\Gamma(d-1)}
\sum_{k=0}^\infty\dfrac{(d)_k}{(5-4d)_k}+ \\
&+\dfrac{\Gamma(1-d)}{(5-5d)}\dfrac{\Gamma(3-3d)\Gamma(2-2d)\Gamma(d-1)}
{\Gamma(3-2d)\Gamma(d)\Gamma(5-4d)}\hyp32{1,3-3d,2-2d}{3-2d,2-d}{1}
\sum_{k=0}^\infty\dfrac{(d)_k}{(5-4d)_k}
\end{align*}
Applying the Thomae transformation \eqref{55.1} and working out the summations we get at last 
\begin{multline*}
c_5(12)=\dfrac{\Gamma(1-d)^3\Gamma(4-5d)}{\Gamma(4-4d)\Gamma(6-5d)}
\dfrac{\Gamma(2-2d)\Gamma(2d-1)}{\Gamma(d)}- \\
-\dfrac{\Gamma(1-d)}{(1-d)}\dfrac{\Gamma(2d-1)}{\Gamma(2d)}
\dfrac{\Gamma(3-3d)\Gamma(4-5d)}{\Gamma(4-4d)\Gamma(6-5d)}
\hyp32{1,d,2d-1}{2-d,2d}{1}
\end{multline*}

\subsection{The fifth cumulant}
For the fifth cumulant we use \eqref{53.1} and \eqref{53.3}. We get
\[
\kappa_5(d)=3840\left(\dfrac{1}{2}(1-2d)(1-d)\right)^{5/2}\sum_{i=1}^{12}c_5(i)
\]
The summation of the $c_5(1..12)$ yields 18 terms. Of these, there are three terms without a hypergeometric function, three terms with a $_4F_3$ hypergeometric function, and twelve terms with a $_2F_1 $ hypergeometric function. The terms with the $_4F_3$ functions cancel each other out, as do two of the terms without the hypergeometric function. A total of six terms remains. The result is
\[
\kappa_5(d)=384\left(\dfrac{1}{2}(1-2d)(1-d)\right)^{5/2}\, 10\, S
\]
with 
\begin{align}
S&\left[\dfrac{\Gamma(4-5d)}{\Gamma(6-5d)}\right]^{-1}=\dfrac{\Gamma(1-d)^4}{\Gamma(4-4d)}+\dfrac{\Gamma(1-d)^3}{(1-d)\Gamma(2-2d)}\dfrac{\Gamma(3-3d)}{\Gamma(4-4d)}
\hyp32{d,1-d,3-3d}{2-d,4-4d}{1}+ \nonumber \\
&+\dfrac{\Gamma(1-d)^2\Gamma(2-2d)}{(1-d)\Gamma(4-4d)}\hyp32{d,1-d,2-2d}{2-d,4-4d}{1}
+\dfrac{\Gamma(1-d)^2}{6(1-d)^2\Gamma(2-2d)}\hyp32{1,d,3-3d}{3-2d,4-3d}{1}+ \nonumber \\
&+\dfrac{2\Gamma(1-d)^2}{3(1-d)^2\Gamma(2-2d)}\hyp32{1,d,3-3d}{2-d,4-3d}{1}
+\dfrac{\Gamma(1-d)^2}{2(1-d)^2\Gamma(2-2d)}\hyp32{1,d,2-2d}{2-d,3-2d}{1}
\label{54.1}
\end{align}
\

with $0 \leq d \leq 0.5$. For $d=0$ we get for the total sum of the integrals: $10\, S=1$. The formula gives the following table of $\kappa_5(d)$ versus $d$. This table gives the same values as the values from Table 4 in \cite{5}. 

\

\begin{table}[!ht]
\centering
\begin{tabular}[t]{|r|r|r|r|r|r|r|r|r|r|r|r|}
	\hline
	$d$&$0$ &0.05&0.10&0.15&0.20&0.25&0.30&0.35&0.40&0.45&0.50\\
	\hline
	$\kappa_5$ &67.88&67.33&65.46&61.92&56.37&48.51&38.32&26.24&13.68&3.563&0\\
	\hline
\end{tabular}
\caption{The values of $\kappa$ as function of $d$. $\kappa_5$ is computed with formula $\eqref{54.1}$.}
\label{Ta2}
\end{table}
In the next figure we show a plot of this table.

\begin{figure}[ht]
\centering
\includegraphics[width=200pt]{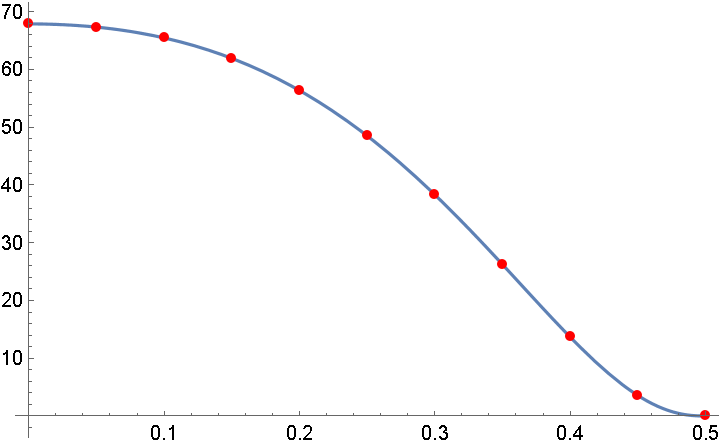}
\caption{   The fifth cumulant $\kappa_5(d)$ versus $d$. The continues line is from $\eqref{54.1}$. The dots are from Table 3.}
\end{figure}
\vspace{0.1 cm}

\section{The method of Veillette and Taqqu for the fourth cumulant}
Veilette and Taqqu designed a method to calculate the cumulants with a recurrence equation. They define the operator
\[
\big(\kappa _{d}f\big)(x)=\int_0^1 |x-u|^{-d}f(u)du
\]
and give the recurrence equation
\begin{equation}
	G_{1,d}(x)=\dfrac{(1-x)^{-d}}{(1-d)^{1/2}} \qquad\qquad G_{k,d}(x)=\big(\kappa _d G_{k-1,d}\big)(x)=\int_0^1 |x-u|^{-d}G_{k-1,d}(u)du
	\label{33.1}
\end{equation}
Then they prove for the factors $c_k$
\[
c_{k}=\int_0^1 G_{\mu,d}(x) G_{\nu ,d}(x) dx
\]
with  $\mu +\nu =k$. 
The cumulants follows then by \eqref{1.3}. For the $k=2,3,4,5$ we get
\begin{align}
	c_2&=\int_0^1 G_{1,d}(x) G_{1 ,d}(x) dx \nonumber \\
	c_3&=\int_0^1 G_{1,d}(x) G_{2 ,d}(x) dx \nonumber \\
	c_4&=\int_0^1 G_{1,d}(x) G_{3 ,d}(x) dx=\int_0^1 G_{2,d}(x) G_{2 ,d}(x) dx
	\label{33.2} \\
	c_5&=\int_0^1 G_{1,d}(x) G_{4 ,d}(x) dx=\int_0^1 G_{2,d}(x) G_{3 ,d}(x) dx
	\label{33.25}
\end{align}
To compute $c_k$ we have to compute $G_{k,d}(x)$. For $k=2$ Veillette and Taqqu found
\[
G_{2,d}(x)=\dfrac{x^{1-d}}{(1-d)^{3/2}}\hyp21{1,d}{2-d}{x}
+\dfrac{\Gamma(1-d)^2}{(1-d)^{1/2}\Gamma(2-2d)}(1-x)^{1-2d}
\]
\vspace{0.1 cm}

To compute $c_4$ with \eqref{33.2} there are two possibilities. We prefer the first possibility because with the second possibility there arises integrals of a product of two hypergeometric functions and these are very difficult. Then we have to compute $G_{3,d}$.

For $k=3$ we get
\begin{align}
	G_{3,d}(x)&=\int_0^1|x-u|^{-d}\, G_{2,d}(u)du \nonumber \\
	&=\dfrac{1}{(1-d)^{3/2}}\int_0^1|x-u|^{-d}u^{1-d}\hyp21{1,d}{2-d}{u}du+ \nonumber \\
	&\qquad\qquad\qquad\qquad\qquad\qquad
	+\dfrac{\Gamma(1-d)^2}{(1-d)^{1/2}\Gamma(2-2d)}\int_0^1|x-u|^{-d}(1-u)^{1-2d}du
	\label{33.2a}
\end{align}
To compute the integrals we use the following Lemmas.
\begin{lemma}
	Let $a,$\, $b,$\, $c,$\, $d,$\, $e,$ and $x$ be real. Let $c>a+b$, $0<x<1$ and $0 \leq d \leq 0.5$. Then 
	\begin{align}
		\int_0^1|x-u]^{-d}&\, u^e \hyp21{a,b}{c}{u}du= \nonumber \\
		&=\Gamma(1-d)\left(\dfrac{\Gamma(1+e)}{\Gamma(2-d+e)}+\dfrac{\Gamma(d-1-e)}{\Gamma(-e)} \right)\, x^{1-d+e}\hyp32{a,b,e+1}{c,2-d+e}{x}+ \nonumber \\
		&+\sum_{m=0}^\infty\dfrac{(a)_m(b)_m}{(c)_m}\dfrac{1}{m!}\dfrac{1}{(1-d+e+m)}
		\hyp21{d-1-e-m,d}{d-e-m}{x}
		\label{L1}
	\end{align}
\end{lemma}
Remark:
\[ \Gamma(1-d)\left(\dfrac{\Gamma(1+e)}{\Gamma(2-d+e)}+\dfrac{\Gamma(d-1-e)}{\Gamma(-e)}\right)=B(1-d,d-1-e)+B(1-d,1+e)
\]
where $B(a,b)$ is the Beta function.
\begin{proof}
	Because of the absolute values we had to split up the integral
	\begin{multline*}
		I(x)=\int_0^1|x-u|^{-d}\, u^e\hyp21{a,b}{c}u{}du= \\
		=\int_0^x (x-u)^{-d}\, u^e\hyp21{a,b}{c}{u}du+
		\int_x^1 (u-x)^{-d}\, u^e\hyp21{a,b}{c}{u}du
	\end{multline*}
	The first integral is known. For the second integral we write the hypergeometric function as a summation and interchange the integral and the summation. The result is
	\begin{align*}
		I(x)=&\dfrac{\Gamma(1+e)\Gamma(1-d)}{\Gamma(2-d+e)}x^{1-d+e}
		\hyp32{a,b,e+1}{c,2-d+e}{x}+ \\
		&+\sum_{k=0}^\infty\dfrac{(a)_k(b)_k}{(c)_k}\dfrac{1}{k!}
		\dfrac{1}{(1-d+e+k)}\hyp21{d-1-e-k,d}{d-e-k}{x}+ \\
		&+\sum_{k=0}^\infty\dfrac{(a)_k(b)_k}{(c)_k}\dfrac{1}{k!}
		x^{1-d+e+k}\dfrac{\Gamma(1-d)\Gamma(d-1-e-k)}{\Gamma(-e-k)}
	\end{align*}
	For the second summation we get
	\begin{align*}
		S&=\sum_{k=0}^\infty\dfrac{(a)_k(b)_k}{(c)_k}\dfrac{1}{k!}
		x^{1-d+e+k}\dfrac{\Gamma(1-d)\Gamma(d-1-e-k)}{\Gamma(-e-k)}=  \\
		&=\Gamma(1-d)\dfrac{\Gamma(d-1-e)}{\Gamma(-e)}
		\sum_{k=0}^\infty\dfrac{(a)_k(b)_k}{(c)_k}\dfrac{1}{k!}x^{1-d+e+k}
		\dfrac{(1+e)_k}{(2-d+e)_k}= \\
		&=\Gamma(1-d)\dfrac{\Gamma(d-1-e)}{\Gamma(-e)}\hyp32{a,b,e+1}{c,2-d+e}{x}
	\end{align*}
	Substitution in $I(x)$ proofs the Lemma
\end{proof}
\begin{lemma}
	Let $d$ and $x$ be real. Let $p$ integer. Let $0<x<1$ and $0 \leq d \leq 0.5$. Then 
	\begin{align*}
		\int_0^1|x-u|^{-d}(1-u)^{p-(p+1)d}du
		&=\dfrac{\Gamma(1-d)\Gamma((p+1)(1-d))}{\Gamma((p+2)(1-d))}(1-x)^{(p+1)-(p+2)d}+ \\
		&+\dfrac{1}{(1-d)}x^{1-d}\hyp21{1,(p+1)d-1}{2-d}{x}
	\end{align*}
\end{lemma}
\begin{proof}
	Because of the absolute values we had to split up the integral
	\begin{multline*}
		\int_0^1|x-u|^{-d}(1-u)^{p-(p+1)d}du= \\
		=\int_0^x(x-u)^{-d}(1-u)^{p-(p+1)d}du+\int_x^1(u-x)^{-d}(1-u)^{p-(p+1)d}du
	\end{multline*}
	These integrals are standard and gives the desired result.
\end{proof}

Application of Lemma  1 to the first integral of \eqref{33.2a} gives
\begin{align*}
	\dfrac{1}{(1-d)^{3/2}}&\int_0^1|x-u]^{-d}\, u^{1-d} \hyp21{1,d}{2-d}{u}du= \\
	&=\dfrac{\Gamma(1-d)}{(1-d)^{3/2}}\left(\dfrac{\Gamma(2-d)}{\Gamma(3-2d)}+
	\dfrac{\Gamma(2d-2)}{\Gamma(d-1)} \right)\, x^{2-2d}\hyp21{1,d}{3-2d}{x}+ \\
	&+\dfrac{1}{(1-d)^{3/2}}\sum_{m=0}^\infty\dfrac{(d)_m}{(2-d)_m}\dfrac{1}{(2-2d+m)}
	\hyp21{2d-2-m,d}{2d-1-m}{x}
\end{align*}

Application of Lemma 2 for the integral in the second term of \eqref{33.2a} gives
\begin{multline*}
	\dfrac{\Gamma(1-d)^2}{(1-d)^{1/2}\Gamma(2-2d)}\int_0^1|x-u|^{-d}(1-u)^{1-2d}du= \\
	=\dfrac{\Gamma(1-d)^2}{(1-d)^{3/2}\Gamma(2-2d)}\, x^{1-d}\hyp21{1,2d-1}{2-d}{x}+
	\dfrac{\Gamma(1-d)^3}{(1-d)^{1/2}\Gamma(3-3d)}(1-x)^{2-3d}
\end{multline*}

Combining these results gives
\begin{align}
	G_{3,d}(x)&=\dfrac{\Gamma(1-d)}{(1-d)^{3/2}}\left(\dfrac{\Gamma(2-d)}{\Gamma(3-2d)}+
	\dfrac{\Gamma(2d-2)}{\Gamma(d-1)} \right)\, x^{2-2d}\hyp21{1,d}{3-2d}{x}+ \nonumber \\
	&+\dfrac{1}{(1-d)^{3/2}}
	\dfrac{\Gamma(2-2d)}{\Gamma(3-2d)}\sum_{k=0}^\infty\dfrac{(d)_k}{( 2-d)_{k}}
	\dfrac{(2-2d)_k}{(3-2d)_k}\hyp21{2d-2-k,d}{2d-1-k}{x}+ \nonumber \\
	&+\dfrac{\Gamma(1-d)^2}{(1-d)^{3/2}\Gamma(2-2d)}\, x^{1-d}\hyp21{1,2d-1}{2-d}{x}+
	\dfrac{\Gamma(1-d)^3}{(1-d)^{1/2}\Gamma(3-3d)}(1-x)^{2-3d}
	\label{33.2b}
\end{align}
With $G_{3,d}(x)$ we can compute $c_4$ with \eqref{33.2}. Substituting in \eqref{33.2} gives for $c_4$
\begin{align*}
	c_{4}&=\dfrac{\Gamma(1-d)}{(1-d)^{3/2}}\left(\dfrac{\Gamma(2-d)}{\Gamma(3-2d)}+
	\dfrac{\Gamma(2d-2)}{\Gamma(d-1)} \right)\int_0^1 x^{2-2d}(1-x)^{-d}\hyp21{1,d}{3-2d}{x}dx+ \\
	&+\dfrac{1}{2(1-d)^3}\sum_{k=0}^\infty\dfrac{(d)_k}{(2-d)_k}
	\dfrac{(2-2d)_k}{(3-2d)_k}\int_0^1(1-x)^{-d}\hyp21{2d-2-k,d}{2d-1-k}{x}dx+ \\
	&+\dfrac{\Gamma(1-d)^2}{(1-d)^{2}\Gamma (2-2d)}\int_0^1x^{1-d}(1-x)^{-d}
	\hyp21{1,2d-1}{2-d}{x}dx+\dfrac{\Gamma(1-d)^3}{(1-d)\Gamma(3-3d)}\int_0^1(1-x)^{2-4d}dx
\end{align*}
All integrals are standard. Integration gives after some rearrangements
\begin{align}
	c_{4}&=\dfrac{1}{2(1-d)^4}\sum_{k=0}^\infty\dfrac{(d)_k}{(2-d)_k}
	\dfrac{(2-2d)_k}{(3-2d)_k}\hyp32{1,d,2d-2-k}{2-d,2d-1-k}{1}+ \nonumber \\
	&+\dfrac{\Gamma(1-d)^3}{(1-d)(3-4d)\Gamma(3-3d)}
	-\dfrac{\Gamma(1-d)^2\Gamma(2d-2)\Gamma(3-2d)}{(1-d)(3-4d)\Gamma(d)\Gamma(3-3d)}+ \nonumber \\
	&+\dfrac{\Gamma(1-d)^4}{(1-d)(3-4d)\Gamma(2-2d)^2}+
	\dfrac{\Gamma(1-d)^3}{(1-d)(3-4d\Gamma(3-3d)}
	\label{33.3}
\end{align}
What remains is the summation in the first term. Applying the Thomae transformation \eqref{55.1} gives
\begin{align*}
	S&=\sum_{k=0}^\infty\dfrac{(d)_k}{(2-d)_k}\dfrac{(2-2d)_k}{(3-2d)_k}
	\hyp32{1,d,2d-2-k}{2-d,2d-1-k}{1}= \\
	&=\sum_{k=0}^\infty\dfrac{(d)_k}{(2-d)_k}\dfrac{(2-2d)_k}{(3-2d)_k}
	\Gamma(2-2d)\Gamma(2d-1-k)\dfrac{1}{\Gamma(-k)}\hyp32{2-2d,1-d,2d-2-k}{-k,2-d}{1}
\end{align*}
With $M$ a non-negative integer there is the well-known transformation
\begin{multline*}
	\dfrac{1}{\Gamma(-M)} \ _{p+1}F_p
	\left(\begin{array}{l}
		a_0,\dots,a_p \\
		-M,b_2,\dots,b_p
	\end{array};x\right)= \\
	=\dfrac{x^{M+1}(a_0)_{M+1}\dots(a_p)_{M+1}}{\Gamma(M+2)(b_2)_{M+1}\dots(b_p)_{M+1}}
	\ _{p+1}F_p
	\left(\begin{array}{l}
		a_0+M+1,\dots,a_p+M+1 \\
		M+2,b_2+M+1,\dots,b_p+M+1
	\end{array};x\right)
\end{multline*}
Application gives
\[
S=\dfrac{(1-d)\Gamma(2d-1)\Gamma(3-2d)}{(2-d)}\dfrac{\Gamma(2d-1)}{\Gamma(2d-2)}
\sum_{k=0}^\infty\dfrac{(d)_k(3-2d)_k}{(2)_k(3-d)_k}
\hyp32{3-2d+k,2-d+k,2d-1}{2+k,3-d+k}{1}
\]
Applying the Thomae transformation \eqref{55.5} gives 
\[
S=\dfrac{4(1-d)^3\Gamma(1-2d)}{(2-d)\Gamma(3-2d)}
\sum_{k=0}^\infty\dfrac{(d)_k}{(3-d)_k}\hyp32{1,d,2d-1}{2d,3-d+k}{1}
\]
Writing the hypergeometric function as a summation over $m$, interchanging the summations and summing over $k$ gives
\begin{equation}
	S=\dfrac{(1-d)\Gamma(1-2d)}{\Gamma(2-2d)}
	\sum_{k=0}^\infty\dfrac{(2d-1)_m(2-2d)_m(d)_m}{(2d)_m(3-2d)_m(2-d)_m}
	\label{33.4}
\end{equation}
The summation can be written as a $_4F_3$ hypergeometric function and we get
\[
S=\dfrac{(1-d)\Gamma(1-2d)}{\Gamma(2-2d)}\hyp43{2d-1,2-2d,1,d}{2d,3-2d,2-d}{1}
\]
A known property of a special $_4F_3$ hypergeometric function is
\begin{equation*}
	\hyp43{a,b,c,d}{a+1,b+1,e}{1}=\dfrac{b(a-e)}{e(a-b)}\hyp32{a,c,d}{a+1,e+1}{1}-
	\dfrac{a(b-e)}{e(a-b)}\hyp32{b,c,d}{b+1,e+1}{1}
\end{equation*}
Application gives
\begin{multline*}
	\hyp43{2d-1,2-2d,1,d}{2d,3-2d,2-d}{1}= \\
	=\dfrac{(2-2d)(3-3d)}{(2-d)(3-4d)}\hyp32{2d-1,1,d}{2d,3-d}{1}+ \dfrac{d\,(1-2d)}{(2-d)(3-4d)}\hyp32{2-2d,1,d}{3-2d,3-d}{1}
\end{multline*}
However doing the summation in \eqref{33.4} with Mathematica  gives
\begin{multline}
	\hyp43{2d-1,2-2d,1,d}{2d,3-2d,2-d}{1}= \\
	=\dfrac{(1-d)}{(3-4d)}\hyp32{2-2d,1,d}{3-2d,2-d}{1}+ \dfrac{2\,(1-d)^2}{(1-2d)(3-4d)}\hyp32{2d-1,1,d}{2d,2-d}{1}
	\label{33.5}
\end{multline}
It is hard to prove that the two results are equal. We leave this to the reader.

Using the Thomae transformation \eqref{55.3} for the second hypergeometric function of \eqref{33.5} gives
\[
S=\dfrac{2(1-d)}{(3-4d)}\hyp32{2-2d,1,d}{3-2d,2-d}{1}+
\dfrac{\Gamma(2-d)^2\Gamma(2d)\Gamma(3-2d)}{(3-4d)(1-2d)\Gamma(3-3d)\Gamma(d)}
\]
Substitution in \eqref{33.3} gives after a lot of manipulations
\begin{align*}
	c_{4}&=\dfrac{1}{(1-d)^3(3-4d)}\hyp32{2-2d,1,d}{3-2d,2-d}{1}+ \\ 
	&+\dfrac{\Gamma(1-d)^4}{(3-4d)(1-d)\Gamma(2-2d)^2}+
	\dfrac{2\Gamma(1-d)^3}{(3-4d)(1-d)\Gamma(3-3d)}
\end{align*}
This is the same result as in Section 4.
\vspace{0.1 cm}

\section{The method of Veillette and Taqqu for the fifth cumulant}
To use the method of Veillette and Taqqu for computing the fifth cumulant we use the first possibility of \eqref{33.25}
\[
c_5=\int_0^1 G_{1,d}(x) G_{4,d}(x) dx=\dfrac{1}{(1-d)^{1/2}}
\int_0^1(1-x)^{-d}\int_0^1|x-u|^{-d}G_{3,d}(u)dudx
\]
For $G_{3,d}(x)$ we use \eqref{33.2b}. The last integral is $G_{4,d}(x)$. We get
\begin{align*}
	G_{4,d}(x)&=\dfrac{\Gamma(1-d)}{(1-d)^{3/2}}\left(\dfrac{\Gamma(2-d)}{\Gamma(3-2d)}+
	\dfrac{\Gamma(2d-2)}{\Gamma(d-1)}\right)I_1(x) \\
	&+\dfrac{1}{(1-d)^{3/2}}
	\dfrac{\Gamma(2-2d)}{\Gamma(3-2d)}\sum_{m=0}^\infty\dfrac{(d)_m}{( 2-d)_m}
	\dfrac{(2-2d)_m}{(3-2d)_m}I_2(x)+ \\
	&+\dfrac{\Gamma(1-d)^2}{(1-d)^{3/2}\Gamma(2-2d)}I_3(x)+
	\dfrac{\Gamma(1-d)^3}{(1-d)^{1/2}\Gamma(3-3d)}I_4(x)
\end{align*}
with
\begin{align*}
	&I_1(x)=\int_0^1|x-u|^{-d}u^{2-2d}\hyp21{1,d}{3-2d}{u}du \qquad
	\ \,  I_2(x)=\int_0^1|x-u|^{-d}\hyp21{2d-2-k,d}{2d-1-k}{u}du  \\
	&I_3(x)=\int_0^1|x-u|^{-d}u^{1-d}\hyp21{1,2d-1}{2-d}{u}du \qquad
	I_4(x)=\int_0^1|x-u|^{-d}(1-u)^{2-3d}du
\end{align*}
For the integrals with the hypergeometric functions we use Lemma 1.
\begin{multline*}
	I_1(x)=\Gamma(1-d)\left(\dfrac{\Gamma(3-2d)}{\Gamma(4-3d)}+\dfrac{\Gamma(3d-3)}{\Gamma(2d-2)} \right)x^{3-3d}\hyp21{1,d}{4-3d}{x}+ \\
	+\sum_{k=0}^\infty\dfrac{(d)_k}{(2-d)_k}\dfrac{1}{(3-3d+k)}\hyp21{3d-3-k,d}{3d-2-k}{x}
\end{multline*}
\begin{multline*}
	I_2(x)=\dfrac{1}{(1-d)}x^{1-d}\hyp32{2d-2-m,d,1}{2d-1-m,2-d}{x} + \\
	+\sum_{k=0}^\infty\dfrac{(2d-2-m)_k (d)_k}{(2d-1-m)_k}
	\dfrac{1}{k!}\dfrac{1}{(3-3d+k)}\hyp21{3d-3-k,d}{3d-2-k}{x}
\end{multline*}
\begin{multline*}
	I_3(x)=\Gamma(1-d)\left(\dfrac{\Gamma(2-d)}{\Gamma(3-2d)}+\dfrac{\Gamma(2d-2)}{\Gamma(d-1)}\right)x^{2-2d}\hyp21{1,2d-1}{3-2d}{x}+ \\
	\sum_{k=0}^\infty\dfrac{(2d-1)_k}{(2-d)_k}\dfrac{1}{(2-2d+k)}\hyp21{2d-2-k,d}{2d-2-k}{x }
\end{multline*}
Application of Lemma 2 for $I_4(x)$ gives
\[
I_4(x)=\dfrac{\Gamma(1-d)\Gamma(3-3d)}{\Gamma(4-4d)}(1-x)^{3-4d}+
\dfrac{1}{(1-d)}x^{1-d}\hyp21{3d-2,1}{2-d}{x} \qquad\qquad\qquad\qquad
\]
For the computation of $c_5$ we had to do the integration of the product of $(1-x)^{-d}$ and $I_{1,2,3,4}(x)$. We get
\begin{align*}
\int_0^1(1-x)^{-d}I_1(x)dx
&=A_1\int_0^1(1-x)^{-d}x^{3-3d}\hyp21{1,d}{4-3d}{x}dx+ \\
&+\sum_{k=0}^\infty B_1(k)
\int_0^1(1-x)^{-d}\hyp21{3d-3-k,d}{3d-2-k}{x}dx \qquad\qquad
\end{align*}
\begin{align*}
\int_0^1(1-x)^{-d}I_2(x)dx
&=\dfrac{1}{(1-d)}x^{1-d}\int_0^1(1-x)^{-d}\hyp32{2d-2-m,d,1}{2d-1-m,2-d}{x}dx + \\
&+\sum_{k=0}^\infty\ B_2(k)\int_0^1(1-x)^{-d}\hyp21{3d-3-k,d}{3d-2-k}{x}dx
\end{align*}
\begin{align*}
\int_0^1(1-x)^{-d}I_3(x)dx
&=A_2\int_0^1(1-x)^{-d}x^{2-2d}\hyp21{1,2d-1}{3-2d}{x}dx+ \\
&+\sum_{k=0}^\infty B_3(k)\int_0^1(1-x)^{-d}\hyp21{2d-2-k,d}{2d-2-k}{x}dx \qquad\qquad
\end{align*}
\begin{align*}
\int_0^1(1-x)^{-d}I_4(x)dx
&=\dfrac{\Gamma(1-d)\Gamma(3-3d)}{\Gamma(4-4d)}\int_0^1(1-x)^{3-5d}dx+ \\
&+\dfrac{1}{(1-d)}\int_0^1(1-x)^{-d}x^{1-d}\hyp21{3d-2,1}{2-d}{x}dx \qquad\qquad
\end{align*}
For $c_5$ we had to add all the integrals. These integrals are all standard. Most of the terms are simple. However there remains four terms with summation of a hypergeometric function. These are very complicated. It is hard to show that this gives the same result as in Section 5. This will therefore not be elaborated further here.
\vspace{0.1 cm}

\section{Discussion}
In this article, two methods are discussed for the calculation of the cumulants of the Rosenblatt distribution. Both methods require integrals to be calculated. The result is a formula with hypergeometric functions. To compare the methods, we look at the number of integrals used to determine $c_k$. The calculation of $c_5$ is discussed as an example.
In the method of Veillette and Taqqu all previous $c_k$ must be known as well. It is assumed that the formula for the third cumulant is known. Furthermore, the Lemmas are used. Then for $c_4$ 4 integrals have to be calculated. For $c_5$ that are 8 integrals. Total 12 integrals. These are all standard integrals.

The direct method does not require all previous $c_k$ to be known. In principle, 60 integrals must be calculated for $c_5$. These are all standard and many integrals are equal. In practice, that will be much less.

In conclusion, it appears that all integrals to be determined in both methods are relatively simple. That leaves the number of summations with hypergeometric functions. These are much more difficult with the method of Veillette and Taqqu than with the direct method. Note that for $c_6$ there are 720 possible regions of which 60 regions remain. This gives 360 integrals!
We are working to have the calculation done automatically by a computer program.
\vspace{0.1 cm}

\section{Appendix: Some Thomae transformations}

For the $_3F_2$ hypergeometric functions with unit argument there are a lot of Thomae transformations \cite{1}. Here are some samples used in this paper. In all formulas we can interchange $a$,$b$ and $c$, as well $e$ and $f$.
\begin{equation}
\hyp32{a,b,c}{e,f}{1}=\dfrac{\Gamma(e+f-a-b-c)\Gamma(f)}{\Gamma(f-a)]\Gamma(e+f-b-c)}
\hyp32{a,e-c,e-b}{e+f-b-c,e}{1}
\label{55.1}
\end{equation}

\begin{multline}
\qquad\qquad\quad\hyp32{a,b,c}{e,f}{1}=\dfrac{\Gamma(e+f-a-b-c)\Gamma(e)\Gamma(f)}
{\Gamma(a)\Gamma(e-f+a-c)\Gamma(e+f-a-b)} \\
\hyp32{e+f-a-b-c,f-a,e-a}{e+f-a-c,e+f-a-b}{1}
\label{55.5}
\end{multline}
\begin{multline}
\qquad\qquad\quad\hyp32{a,b,c}{e,f}{1}=\dfrac{\Gamma(1-c)\Gamma(f)\Gamma(a+1)\Gamma(b-a)}
{\Gamma(b)\Gamma(f-a)\Gamma(a-c+1)}- \\
-\dfrac{a\, \Gamma(1-c)\Gamma(f)}{\Gamma(b-c+1)\Gamma(f-b)(b-a)}
\hyp32{b,b-f+1,b-a}{b-a+1,b-c+1}{1}
\label{55.3}
\end{multline}

\end{document}